\documentclass[12pt,a4paper]{amsart}
\usepackage[marginratio=1:1]{geometry}

\usepackage[T1]{fontenc}
\usepackage[utf8]{inputenc}
\usepackage[american]{babel}
\usepackage[draft=false]{hyperref}
\usepackage{enumitem}
\usepackage{graphicx}
\usepackage{verbatim}

\usepackage{floatrow}

\usepackage{xcolor}
\hypersetup{
	colorlinks,
	linkcolor={red!50!black},
	citecolor={blue!50!black},
	urlcolor={blue!80!black}
}

\usepackage{amsfonts}
\usepackage{amssymb}
\usepackage{tikz-cd}
\usepackage{float}
\makeatletter
\providecommand*{\twoheadrightarrowfill@}{%
	\arrowfill@\relbar\relbar\twoheadrightarrow
}
\providecommand*{\twoheadleftarrowfill@}{%
	\arrowfill@\twoheadleftarrow\relbar\relbar
}
\providecommand*{\xtwoheadrightarrow}[2][]{%
	\ext@arrow 0579\twoheadrightarrowfill@{#1}{#2}%
}
\providecommand*{\xtwoheadleftarrow}[2][]{%
	\ext@arrow 5097\twoheadleftarrowfill@{#1}{#2}%
}
\makeatother

\newcommand{\C}{{\mathfrak C}}
\newcommand{\M}{{\mathcal M}}

\newcommand{\R}{{\mathbb{R}}}

\DeclareMathOperator{\tp}{{tp}}
\DeclareMathOperator{\Th}{{Th}}
\DeclareMathOperator{\gal}{{Gal}}

\DeclareMathOperator{\aut}{{Aut}}
\DeclareMathOperator{\autf}{{Autf}}

\DeclareMathOperator{\st}{{st}}
\DeclareMathOperator{\bdd}{{bdd}}
\DeclareMathOperator{\acl}{{acl}}
\DeclareMathOperator{\Fix}{{Fix}}

\DeclareMathOperator{\St}{{St}}
\DeclareMathOperator{\Def}{{Def}}

\def\lL{\mathcal L}

\newtheorem{thm}{Theorem}[section]

\newtheorem{ques}[thm]{Question}

\newtheorem{lem}[thm]{Lemma}
\newtheorem{fct}[thm]{Fact}
\newtheorem{cor}[thm]{Corollary}
\newtheorem{prop}[thm]{Proposition}

\theoremstyle{remark}
\newtheorem{rem}[thm]{Remark}
\theoremstyle{definition}
\newtheorem{dfn}[thm]{Definition}

\newtheorem{clm*}{Claim}
\newtheorem{ex}[thm]{Example}
\newcounter{claimcounter}[thm]
\newenvironment{clm}{\stepcounter{claimcounter}{\noindent {\textbf{Claim}} \theclaimcounter:}}{}
\newenvironment{clmproof}[1][\proofname]{\proof[#1]}{\endproof}

\title{On first order amenability}
\author{Ehud Hrushovski}
\email[E. Hrushovski]{Ehud.Hrushovski@maths.ox.ac.uk}
\address[E.\ Hrushovski]{Mathematical Institute, University of Oxford\\
Andrew Wiles Building\\
Radcliffe Observatory Quarter (550)\\
Woodstock Road\\
Oxford OX2 6GG, UK}

\author{Krzysztof Krupi\'nski}
\email[K.\ Krupi\'{n}ski]{kkrup@math.uni.wroc.pl}
\address[K.\ Krupi\'nski]{
	Instytut Matematyczny, Uniwersytet Wroc\l awski\\
	pl. Grunwaldzki 2\\
	50-384 Wroc\l aw, Poland
}
\thanks{The second author is supported by National Science Center, Poland, grants 2015/19/B/ST1/01151, 2016/22/E/ST1/00450, and 2018/31/B/ST1/00357}
\address{
	ORCID (K.\ Krupi\'{n}ski): \href{https://orcid.org/0000-0002-2243-4411}{0000-0002-2243-4411}}

\author{Anand Pillay}
\email[A.\ Pillay]{apillay@nd.edu}
\address[A.\ Pillay]{Department of Mathematics, University of Notre Dame\\
	255 Hurley Hall\\
	Notre Dame, IN 46556, USA}
\thanks{The third author was supported by NSF grants DMS-136702,  DMS-1665035, and DMS-1760212}

\keywords{Amenable theory, $G$-compactness}

\subjclass[2010]{03C45, 43A07}

\date{}

\begin{document}
\maketitle

\begin{abstract}
%We introduce  the notion of {\em first order [extreme] amenability}, as a property of a first order theory $T$:  every complete type over $\emptyset$, in possibly infinitely many variables, extends to an automorphism-invariant global Keisler measure [type] in the same variables. [Extreme] amenability of $T$ will follow from [extreme] amenability of the (topological) group $\aut(M)$ for all sufficiently large $\aleph_{0}$-homogeneous countable models $M$ of $T$ (assuming $T$ to be countable), but is radically less restrictive. 
We introduce  the notion of {\em first order  amenability}, as a property of a first order theory $T$:  every complete type over $\emptyset$, in possibly infinitely many variables, extends to an automorphism-invariant global Keisler measure  in the same variables.   Amenability of $T$   follows from  amenability of the (topological) group $\aut(M)$ for all sufficiently large $\aleph_{0}$-homogeneous countable models $M$ of $T$ (assuming $T$ to be countable), but is radically less restrictive.

First, we study basic properties of amenable theories, giving many equivalent conditions. Then, applying a version of the stabilizer theorem from \cite{HKP1}, we prove that if $T$ is amenable, then $T$ is $G$-compact, namely Lascar strong types and Kim-Pillay strong types over $\emptyset$ coincide.  This extends and essentially generalizes a similar result proved via different methods for $\omega$-categorical theories in \cite{KrPi}. In the special case when amenability is witnessed by $\emptyset$-definable global Keisler measures (which is for example the case for amenable $\omega$-categorical theories), we also give a different proof, based on stability in continuous logic.

Parallel (but easier) results hold for the notion of {\em extreme amenability}.
%%udi
\end{abstract}

\section{Introduction}

%The general motivation standing behind this research is to understand relationships between dynamical properties of groups of automorphisms of first order structures and model-theoretic properties of the underlying theories. More specifically, similarly to Section 4 of \cite{KrPi}, in this paper our goal is to understand model-theoretic consequences of certain natural amenability conditions which we introduce and explore.

We introduce the notions of {\em amenable} and {\em extremely amenable} first order theory.  This is part of our attempt to extract the model-theoretic content of the circle of ideas around [extreme] amenability of automorphism groups of countable structures, which we discuss further below.  We say that $T$ is {\em amenable} if for every $p \in S_{\bar x}(\emptyset)$, in any (possibly infinite) tuple of variables $\bar x$, there exists an $\aut(\C)$-invariant, Borel probability measure on $S_p(\C):=\{ q \in S_{\bar x}(\C): p \subseteq q\}$, where $\C$ is a monster model of $T$. {\em Extreme amenability} of $T$ means that the invariant measure above can be chosen to be a {\em Dirac}, namely: every $p \in S_{\bar x}(\emptyset)$ extends to a global $\aut(\C)$-invariant complete type.  We study properties of [extreme] amenability, showing for example that they are indeed  properties of the theory (i.e. do  not depend on $\C$) and providing several equivalent definitions. We will discuss here amenability, leaving the extreme version to further paragraphs. One of the equivalent definitions of amenability of $T$  is that $\aut(\C)$ is {\em relatively amenable} (i.e. there is an $\aut(\C)$-invariant, finitely additive, probability measure on the Boolean algebra of relatively definable subsets of $\aut(\C)$ treated as a subset of $\C^{\C}$). Relative amenability of $\aut(\C)$ (or, more generally, of the group of automorphisms of any model) is a natural counterpart of definable amenability of a definable group. The above observations work for any $\aleph_0$-saturated and strongly $\aleph_0$-homogeneous model $M$ in place of $\C$. For such an $M$, if $\aut(M)$ is amenable as a topological group (with the pointwise convergence topology), then $T$ is amenable. We point out in a similar fashion that (for countable $T$) if $\aut(M)$ is amenable for all sufficiently large $\aleph_{0}$-homogeneous countable models, then $T$ is amenable.  In the NIP context, we get a full characterization of amenability of $T$ in various terms, e.g. by saying that $\emptyset$ is an extension  base, which also yields a class of examples of amenable theories, e.g. all stable or o-minimal or $c$-minimal theories are amenable. Also, the theories  of measurable structures in the sense of Elwes and Macpherson (e.g. pseudo-finite fields) \cite{ElMa} are amenable. 
%%%Krzys: Maybe somewhere here we should say briefly about G-compactness and why it is important. Can you add something like that? The next paragraph added by Anand. (I made minor changes in it.)

%%%Krzys: I changed "this" to "those".
This paper is concerned with the implications of [extreme] amenability of a first order theory $T$ for the Galois group $\gal_{L}(T)$. So let us discuss briefly those Galois groups as well as the notions of $G$-compactness and $G$-triviality and why they should be considered important.  Formal definitions will be given in Section \ref{section: preliminaries on G-compactness}, but we give a rather more relaxed description now.  See also the introduction to \cite{KrPi}.
%%%Krzys: Instead of "(finite or $\omega$)" I wrote "(bounded)".
At the centre are the key notions of {\em strong types}. Two tuples $\bar a$ and $\bar b$ from the monster model ${\mathfrak C}$, of the same (bounded) length, 
%(finite or $\omega$) 
have the same {\em Lascar strong type} if $E(\bar a,\bar b)$ whenever $E$ is an $\aut({\mathfrak C})$-invariant equivalence relation with boundedly many classes. 
%%%Krzys: I wrote "Kim-Pillay" instead of "KP" and added "(shortly, KP-strong type)".
If we instead consider only bounded equivalence relations $E$ which are {\em type-definable} over $\emptyset$, we obtain the notion of having the same {\em Kim-Pillay strong type} (in short, KP-strong type). 
%%%Krzys: I added "Galois", as I prefer to call $Gal_L(T)$ the "Lascar Galois group" and we use this terminology elsewhere.
 The group of permutations of all Lascar strong types induced by $\aut({\mathfrak C})$ is called the {\em Lascar Galois group} $\gal_{L}(T)$;
%%%Krzys: I added the next sentence.
$\gal_{KP}(T)$ is defined analogously. 
When Lascar strong types coincide with KP-strong types, $\gal_{L}(T)$ has naturally the structure of a compact Hausdorff group, and $T$ is said to be 
{\em $G$-compact}.  When Lascar strong types coincide with types (over $\emptyset$), then $\gal_{L}(T)$ is trivial, and $T$ is said to be {\em $G$-trivial}. Lascar strong types present {\em obstructions} to various kinds of type amalgamation.  Also in \cite{La}, where the Lascar Galois group was first defined, they present obstacles to recovering an $\omega$-categorical theory $T$ from its category of models. 
%%%Krzys: I slightly modified the next sentence.
As KP-strong types are much easier to handle than Lascar strong types, $G$-compactness is a desirable property.  In any case, $\gal_{L}(T)$ and $\gal_{KP}(T)$ are important invariants of an arbitrary complete first order theory $T$ and play important  roles in model theory.

The main result of this paper (proved in Section \ref{Subsection 4.2}) is the following
%
%{\bf Theorem.} {\em Every amenable theory is G-compact.}\\[3mm]  
%
\begin{thm}\label{The main theorem}
Every amenable theory is $G$-compact.
\end{thm}

%%%Krzys Dec 2020: I replaced ``essentially generalizes'' by ``is a wide generalization of'', as the referee did not understand ``essentially generalizes''.
This result is a wide generalization of Theorem 0.7 from \cite{KrPi} which says that whenever $M$ is a countable, $\omega$-categorical structure and $\aut(M)$ is amenable as a topological group, then $\Th(M)$ is $G$-compact. Theorem 0.7 of \cite{KrPi} was deduced (by a non-trivial argument which is interesting in its own right) from \cite[Theorem 0.5]{KrPi}, more precisely, from the fact that amenability of a topological group implies equality of certain model-theoretic/topological connected  components. In \cite{KrPi}, this last fact was proved for groups possessing  a basis of open neighborhoods of the identity consisting of open subgroups, which was sufficient in the proof of \cite[Theorem 0.7]{KrPi}, because $\aut(M)$ has this property; later, this fact was proved in full generality in \cite[Corollary 2.37]{HKP1}.  As to our very general Theorem \ref{The main theorem}, we do not have an argument showing that it follows from \cite[Corollary 2.37]{HKP1}; 
%%%Krzys Jan 2021: I wrote ``relatively type-definable subsets'' instead of ``relatively definable subsets''.
instead we give a direct proof working with {\em relatively type-definable subsets} of the group of automorphisms of the monster model and using  a version from \cite{HKP1} of Massicot-Wagner stabilizer theorem \cite{MaWa}.
%The engine of our proof is once again the argument from \cite{MaWa}.
%%%Krzys Jan 2021: I added Udi's text below in a modified form. I hope this is what Udi had in mind (talking about MW). If not, please correct.
Theorem \ref{The main theorem} can be viewed as a transposition of \cite{MaWa} from definable
groups to theories.  It might have followed easier from an application of
\cite{MaWa} to the automorphism group of the monster model,  if  the ``stabilizers'' produced in \cite{MaWa} were quantifier-free definable.  
%But on the contrary MW is quantifier-intensive (in the definition of the rank), and so we have to proceed differently. 
But they are not, and so we have to proceed differently.

%%%Krzys Dec 2020: I added ``(see Proposition \ref{proposition: main theorem for definable measures})''. Instead of ``simpler proof'' I wrote ``completely different proof'', I added `(with a better bound on the diameters of Lascar strong types than the one obtained in Theorem 4.1)`'', and I permuted the rest of the sentence. I also added the last sentence in this paragraph.
%%%Krzys Jan 2021: So I added ``(in which case we say that the theory is {\em definably amenable})''. 
In Section \ref{Subsection 4.1 and a half} (see Proposition \ref{proposition: main theorem for definable measures}), we give a completely different proof of Theorem \ref{The main theorem} (with a better bound on the diameters of Lascar strong types than the one obtained in Theorem \ref{theorem: relative definable amenability implies G-compactness}) which is based on stability theory in continuous logic,  but under the stronger assumption of the existence of $\emptyset$-definable Keisler measures on all $\emptyset$-definable sets (in which case we say that the theory is {\em definably amenable}).  This also includes the $\omega$-categorical context from \cite[Theorem 0.7]{KrPi}, yielding yet another proof of \cite[Theorem 0.7]{KrPi}. The readers who do not feel comfortable with continuous logic can skip Section \ref{Subsection 4.1 and a half} with no harm.

%%%Krzys Dec 2020: I added the next sentence, as the referee asked about the converse.
Let us note that the converse of Theorem \ref{The main theorem} does not hold: for example, the theory of a dense circular order is known to be $G$-compact, but it is not amenable, because it has NIP and $\emptyset$ is not an extension base.

Extreme amenability of automorphism groups of (arbitrary) countable structures $M$ was studied in detail by Kechris, Pestov, and Todor\'{c}evi\'{c}.  Their paper \cite{KPT} inspired a whole school, connecting to structural Ramsey combinatorics and dynamics. When $\Th(M)$ is $\omega$-categorical, then extreme amenability of $\aut(M)$ is a property of this first order theory, so is a model-theoretic notion  (in the sense of model theory being the study of first order theories rather than arbitrary structures).  Some of this extends to homogeneous models of arbitrary theories and to continuous logic (thanks to Todor Tsankov for a conversation about this with one of the authors). 

%%%Krzys Dec 2020: The referee had problems to understand the next paragraph. He even suggested moving it to further sections. I did not do that, because I see it as a motivation for our notion of extreme amenability. Instead, I added some explanations in the second half of the paragraph below (which was vague for the referee). The changes are described below.
Let us comment on the relation between extreme amenability of the automorphism group of an $\omega$-categorical, countable structure $M$ as considered in \cite{KPT} (which we call  KPT-extreme amenability) and extreme amenability of $\Th(M)$ in our sense. KPT-extreme amenability concerns {\em all} flows of the topological group $\aut(M)$ and says that the universal flow (or rather ambit) has a fixed point.  Our first order extreme amenability (of $\Th(M)$) can also be read off from flows of $\aut(M)$ and says that a {\em particular flow} $S_{\bar m}(M)$ has a fixed point (where $\bar m$ is an enumeration of $M$ and $S_{\bar m}(M)$ here denotes the space of complete extensions of $\tp({\bar m})$ over $M$).
% I do not understand Udi's comments about externally definable sets, universal theories,.....
The class of KPT-extremely amenable, $\omega$-categorical theories $T$  is not at present explicitly classified, but appears to be   special (perhaps analogous to monadic stability in the stable world). 
%udi slightly moderated the above sentence %%udi
%%%Krzys Dec 2020: The next sentence is slightly changed. Also an explanation is added in the following sentence.
%It follows from their definition that whenever $\lL'$ is a language extending the language $\lL$ of $T$ and $T'$ is a universal $\lL'$-theory consistent with $T$, then the countable model $M$ of $T$ has an expansion to a model of $T'$ where the new symbols in $\lL'$ are interpreted as certain $\emptyset$-definable sets in $M$.
Note that for an ($\omega$-categorical) KPT-extremely amenable theory $T$, whenever $\lL'$ is a countable language extending the language $\lL$ of $T$ and $T'$ is a universal $\lL'$-theory consistent with $T$, then the countable model $M$ of $T$ has an expansion to a model of $T'$ in which the new symbols in $\lL'$ are interpreted as certain $\emptyset$-definable sets in $M$. Indeed, such an expansion is just a fixed point of the action of $\aut(M)$ on the compact (and non-empty) space of the expansions of $M$ to the models of $T'$.
%%%Krzys Dec 2020: ``extreme amenability'' instead of ``amenability''. I also removed ``Note ... that''.
In particular, KPT-extreme amenability of an $\omega$-categorical structure $M$  implies the existence of a $\emptyset$-definable linear ordering on  $M$.
%%%Krzys Dec 2020: I changed the next sentence. In particular, I explicitly added the assumption of omega-categoricity in one place, and added some explanations in parentheses. Could you give a reference for the very last statement ``this property has also been useful in various situations, such as for elimination of imaginaries.'' The referee asked for references.
%By contrast, our first-order extreme amenability is a quite common property; in particular, all Fra\"{i}ss\'{e} classes with free (or, more generally, canonical) amalgamation enjoy it; so does $T$ expanded by constants for a model, or, when $T$ is stable, for an algebraically closed set in $T^{eq}$, and often also when $T$ is NIP. 
By contrast, our first-order extreme amenability is a quite common property: if the  Fra\"{i}ss\'{e} limit of a Fra\"{i}ss\'{e} class with free (or, more generally, canonical) amalgamation is $\omega$-categorical, then its theory is extremely amenable (see the discussion after Corollary \ref{corollary: amenability implies relative definable amenability}); also, every theory $T$ expanded by constants for a model is extremely amenable (here, coheir extensions over this model are the required invariant types); similarly, every stable $T$ expanded by constants for an algebraically closed set in $T^{eq}$ is extremely amenable (non forking extensions witness it by stationary of all types over $\acl^{eq}(\emptyset)$); there are also many NIP or simple theories (e.g. the random graph) which are extremely amenable. 
%udi
 Although not explicitly named as extreme amenability, the property of 
 % or identified, this property has also been useful in various situations, such as for elimination of imaginaries.
 %%%Krzys Dec 2020 final: I removed  ``extreme amenability'' after the second comma below which I also removed. Is it OK?
 extendibility of types to invariant types
 %, extreme amenability 
 has been frequently considered in the literature, and used notably for the elimination of imaginaries; see e.g.  \cite{HHM}.
%%udi 
%What about Udi's additions which I still do not understand  

Keisler measures play a big role in this paper (especially in the notion of first order amenability) and we generally assume that the reader is familiar with them. A Keisler measure on a sort (or definable set) $X$ over a model $M$ is simply a finitely additive (probability) measure on the Boolean algebra of definable (over $M$) subsets of $X$. As such it is a natural  generalization of a complete type over $M$ containing the formula defining $X$. As pointed out at the beginning of Section 4 of \cite{HrPi}, a Keisler measure on $X$ over $M$ is the ``same thing'' as a regular Borel probability measure on the space $S_{X}(M)$ of complete types over $M$ containing the formula defining $X$.  Keisler measures are completely natural  in model theory, but it took some time for them to be studied systematically.  They were introduced in Keisler's seminal paper \cite{Keisler} mainly in a stable and NIP environment, and later played an important role in \cite{HrPePi} in the solution of some conjectures relating o-minimal groups to compact Lie groups.\\

%%%Krzys: I added the sentence below. I am not sure if it should be included.
%This paper contains the material from Section 4 our preprint ``Amenability and definability'' which we decided to divide into two papers.
This paper contains the material in Section 4 of our preprint ``Amenability and definability''. Following the advice of editors and referees we have divided that preprint into two papers: the first being \cite{HKP1}, and the second being the current paper.

Since this manuscript first appeared on arXiv, a number of other works and conjectures appeared exploring the transposition from definably amenable definable groups to amenable theories or issues specific just for amenable theories or some notions introduced in this paper, e.g. see \cite{CHKKMPR}, \cite{Hr1}, \cite[Section 7]{Hr2}, \cite{KLM}, \cite{KP}, \cite{GHK}, and open questions there. Let us only remark here that  in \cite{CHKKMPR} we constructed the first example of a non definably amenable group in a simple theory which by ``adding an affine sort construction'' yielded an example of a nonamenable simple theory. The above paper also reveals a different behavior of  definable amenability of groups and theories: every group definable in a small theory is definably amenable but not every small theory is amenable.

%\section{Relative definable amenability and G-compactness}

%\section{Amenability and G-compactness}\label{Section 4}

%In this section, we introduce the notions of amenable types and theories, and study their consequences. In Subsection \ref{Subsection 4.1}, we analyze various basic properties and provide equivalent definitions. In Subsections \ref{Subsection 4.1 and a half} and \ref{Subsection 4.2}, we prove the main result of this section that amenability implies $G$-compactness: in Subsection \ref{Subsection 4.1 and a half}, we give a relatively simpler proof assuming the existence of  definable measures, and in Subsection \ref{Subsection 4.2}, we give a proof in full generality, using in particular Corollary \ref{corollary: good for applications} and arguments in the spirit of Section \ref{Section 2} (although in a different context).

%%%Krzys: This section is just Subsection 4.1 of our preprint. I only added the second paragraph.

\section{Preliminaries on $G$-compactness}\label{section: preliminaries on G-compactness}

We only recall a few basic definitions and facts about Lascar strong types and Galois groups. For more details the reader is referred to  \cite{LaPi}, \cite{CLPZ} or \cite{Zi}.

%For more detailed expositions the reader is referred to \cite[Subsection 1.3]{KrPiRz} or \cite[Subsection 4.1]{KrPi_recent}. If the reader is interested in yet more details and proofs, he or she may consult fundamental papers around this topic, e.g. \cite{LaPi} or \cite{CLPZ}.
 
As usual, by a monster model of a given complete  theory we mean a $\kappa$-saturated and strongly $\kappa$-homogeneous model for a sufficiently large cardinal $\kappa$ (typically, $\kappa>|T|$ is a strong limit cardinal).  Where recall that the (standard) expression ``strongly $\kappa$-homogeneous'' means that any partial elementary map between subsets of the model of cardinality $<\kappa$ extends to an automorphism of the model. 
A set [tuple] is said to be small [short] if it is of bounded cardinality (i.e. $<\kappa$). 

Let $\C$ be a monster model of a complete theory $T$.

\begin{dfn}\label{definition: Gal_L}$\,$
		\begin{enumerate}[label=\roman{*}),nosep]
			\item
			{\em The group of Lascar strong automorphisms}, which is denoted by $\autf_L(\C)$, is the subgroup of $\aut(\C)$ which is generated by all automorphisms fixing a small submodel of $\C$ pointwise, i.e.\ $\autf_L(\C)=\langle \sigma : \sigma \in \aut(\C/M)\;\, \mbox{for a small}\;\, M\prec \C\rangle$.
			\item
			{\em The Lascar Galois group of $T$}, which is denoted by $\gal_L(T)$, is the quotient group $\aut(\C)/\autf_L(\C)$ (which makes sense, as $\autf_L(\C)$ is a normal subgroup of $\aut(\C)$). It turns out that $\gal_L(T)$ does not depend on the choice of $\C$ (e.g. see \cite[Fact 4.2]{LaPi}).
		\end{enumerate}
	\end{dfn}

The orbit equivalence relation of $\autf_L(\C)$ acting on any given product $S$ of boundedly (i.e. less than the degree of saturation of $\C$) many sorts of $\C$ is usually denoted by $E_L$. It turns out that this is the finest bounded (i.e. with boundedly many classes), invariant equivalence relation on $S$ (see \cite[Proposition 5.4]{KiPi}); and the same is true after the restriction to the set of realizations of any type in $S(\emptyset)$ or even to any invariant set. The classes of $E_L$ are called {\em Lascar strong types}. It turns out that $\autf_L(\C)$ coincides with the the group of all automorphisms fixing setwise all $E_L$-classes on all (possibly infinite) products of sorts. 
%%%Krzys Dec 2020: I added the next sentence. It is not exactly what the referee suggested, but I think it is enough to see that the definitions from the introduction are equivalent to the ones given above.
So we see that $\gal_L(T)$ can be identified with the group of elementary (i.e. induced by $\aut(\C)$) permutations of all Lascar strong types (as written in the introduction).

%%%Krzys Dec 2020:  \tp(\bar m /\emptyset) instead of  \tp(\bar m /M)
For any small $M \prec \C$ enumerated as $\bar m$, we have a natural surjection from $S_{\bar m}(M):=\{ p \in S(M): \tp(\bar m /\emptyset) \subseteq p\}$ to $\gal_L(T)$ given by $\tp(\sigma(\bar m)/M) \mapsto \sigma/\autf_L(\C)$ for $\sigma \in \aut(\C)$. We can equip $\gal_L(T)$ with the quotient topology induced by this surjection, and it is easy to check that this topology does not depend on the choice of $M$. 
%%%Krzys Dec 2020: I slightly changed the next sentence (using ``quasi-compact'')
%In this way, $\gal_L(T)$ becomes a topological (but not necessarily Hausdorff) group (see \cite{Zi} for a detailed exposition).
In this way, $\gal_L(T)$ becomes a quasi-compact (so not necessarily Hausdorff) topological group (see \cite{Zi} for a detailed exposition).

%The group $\gal_L(T)$ can be equipped with the topology making it a topological (but not necessarily Hausdorff) group (see \cite{Zi} for the definition and a detailed exposition).

\begin{dfn}\label{definition: Gal_KP}$\,$
		\begin{enumerate}[label=\roman{*}),nosep]
			\item By $\gal_0(T)$ we denote the closure of the identity in $\gal_L(T)$.
\item
			{\em The group of Kim-Pillay strong automorphisms}, which is denoted by $\autf_{KP}(\C)$, is the preimage of $\gal_0(T)$ under the quotient homomorphism $\aut(\C) \to \gal_L(T)$.
			\item
			{\em The Kim-Pillay Galois group of $T$}, which is denoted by $\gal_{KP}(T)$, is the quotient group $\gal_L(T)/\gal_0(T) \cong \aut(\C)/\autf_{KP}(\C)$ equipped with the quotient topology. It is a compact, Hausdorff topological group.
		\end{enumerate}
	\end{dfn}
 
The orbit equivalence relation of $\autf_{KP}(\C)$ acting on any given product $S$ of (boundedly many) sorts of $\C$ is usually denoted by $E_{KP}$. It turns out that this is the finest bounded (i.e. with boundedly many classes), type-definable over $\emptyset$ equivalence relation on $S$; and the same is true after the restriction to the set of realizations of any type in $S(\emptyset)$ (see \cite[Lemma 4.18]{LaPi}). The classes of $E_{KP}$ are called {\em Kim-Pillay strong types}. It turns out that $\autf_{KP}(\C)$ coincides with the the group of all automorphisms fixing setwise all $E_{KP}$-classes on all (possibly infinite) products of sorts.
%%%Krzys Dec 2020: As above, I added the next sentence to compare with the introduction.
So we see that $\gal_{KP}(T)$ can be identified with the group of elementary permutations of all Kim-Pillay strong types (as written in the introduction).

The theory $T$ is said to be {\em $G$-compact} if the following {\em equivalent} conditions hold.

\begin{enumerate}
\item $\autf_{L}(\C)=\autf_{KP}(\C)$.
\item $\gal_L(T)$ is Hausdorff.
\item Lascar strong types coincide with Kim-Pillay strong types on any (possibly infinite) products of sorts.
\end{enumerate}

%%%Krzys Dec 2020: The referee asked for precise references for the above equivalences. Instead of references, I added the trivial proof below.
Let us briefly explain why the above conditions are equivalent. (1) $\leftrightarrow$ (2) follows from Definitions \ref{definition: Gal_L} and \ref{definition: Gal_KP}. (1) $\rightarrow$ (3) follows from the above definitions of $E_L$ and $E_{KP}$ as the orbit equivalence relations of $\autf_L(\C)$ and $\autf_{KP}(\C)$, respectively. Finally, (3) $ \rightarrow$ (1)  holds, because $\autf_L(\C)$ and $\autf_{KP}(\C)$ are the kernels of the actions of $\aut(\C)$ on the Lascar and Kim-Pillay strong types, respectively.

%%%Krzys Dec 2020 final: I added the next sentence, as ss Udi suggested.
Lascar's definition of $G$-compactness from \cite{La} corresponds in our terminology to saying that $T$ remains $G$-compact after naming any finite set of parameters. Example \ref{example: naming parameters kills G-compactness} yields a $G$-compact theory with a non $G$-compact expansion by a single constant.

By the definition of $E_L$, we see that $\bar \alpha \, E_L\, \bar \beta$ if and only if there are $\bar \alpha_0=\bar \alpha, \bar \alpha_1,\dots, \bar \alpha_n = \bar \beta$ and models $M_0,\dots,M_{n-1}$ such that 
$$\bar \alpha_0 \equiv_{M_0} \bar \alpha_1 \equiv_{M_1}  \dots \bar \alpha_{n-1} \equiv_{M_{n-1}} \bar \alpha_n.$$ 
In this paper, by the {\em Lascar distance} from $\bar \alpha$ to $\bar \beta$  (denoted by $d_L(\bar \alpha,\bar \beta)$) we mean the smallest natural number $n$ as above. By the {\em Lascar diameter} of a Lascar strong type $[\bar \alpha]_{E_L}$ we mean the supremum of $d_L(\bar \alpha, \bar \beta)$ with $\bar \beta$ ranging over $[\bar \alpha]_{E_L}$. It is well known (proved in \cite{Ne}) that $[\bar \alpha]_{E_L}=[\bar \alpha]_{E_{KP}}$ if and only if the Lascar diameter of $[\bar \alpha]_{E_L}$ is finite.

%%%Krzys Dec 2020: I added this (in my opinion redundant) paragraph, as the referee asked for an explanation of the notation $\varphi(\bar x)$ for a formula with infinitely many variables, and the notation $\pi_1(\bar x) \vdash \pi_2(\bar x)$ was confusing for him.
%Throughout this paper, tuples of variables are often infinite; in particular, $\varphi(\bar x)$ (where $\bar x$ is allowed to be infinite) means that the formula $\varphi(\bar x)$ has variables $\bar x$ (among which, of course, only finitely many appear). By a {\em finitary} type, we mean a type in finitely many variables. For partial types $\pi_1(\bar x)$ and $\pi_2(\bar x)$, we will write $\pi_1(\bar x) \vdash \pi_2(\bar x)$ for the implication of types with respect to the underlying theory $T$.
%udi 
Throughout this paper, tuples of variables are often infinite; in particular, $\varphi(\bar x)$ means that the  
free variables of the  formula $\varphi$ are among those listed in $\bar{x}$.   It can be convenient to allow $\bar{x}$
to be infinite, though of course $\varphi$ has only  finitely many free variables.  
 By a {\em finitary} type, we mean a type in finitely many variables.  We generally have a fixed underlying complete theory $T$ 
 in the background;  a partial type for $T$ can be assumed to include the sentences of $T$.  In any case   
for partial types $\pi_1(\bar x)$ and $\pi_2(\bar x)$,  
  $\pi_1(\bar x) \vdash \pi_2(\bar x)$ means by definition that $T \cup \pi_1(\bar x)$ logically implies $ \pi_2(\bar x)$.

%%udi

%\subsection{Relative definable amenability}\label{Subsection 4.1}
%\subsection{Amenable theories}\label{Subsection 4.1}

%%%Krzys: The next section is Subsection 4.2 from the preprint.

\section{Amenable theories: definitions and basic results}\label{Subsection 4.1}

As usual, $\C$ is a monster model of an arbitrary complete theory $T$. Let $\bar c$ be an enumeration of $\C$ and let $S_{\bar c}(\C)=\{ \tp(\bar a/\C) \in S(\C): \bar a \equiv \bar c\}$. More generally, for a partial type $\pi(\bar x)$ over $\emptyset$, put $S_\pi(\C) =\{ q(\bar x) \in S(\C) : \pi \subseteq q\}$. 
If $p(\bar x) \in S(\emptyset)$ and $\bar \alpha \models p$, then $S_{\bar \alpha}(\C):=S_p(\C) =\{ q(\bar x) \in S(\C) : p \subseteq q\}$. (Note that we allow here tuples $\bar x$ of unbounded length (i.e. greater than the degree of saturation of $\C$). Each $S_\pi(\C)$ is naturally an $\aut(\C)$-flow, i.e.
%(i.e. a compact space on which $\aut(\C)$ acts by homeomoprhisms; in fact, this action is continuous if we equip $\aut(\C)$ with the pointwise convergence (equivalently, product) topology).
a compact space together with a continuous action of the group $\aut(\C)$ equipped with the pointwise convergence (equivalently, product) topology.

Let us start from the local version of amenability.

\begin{dfn}\label{definition: amenability of types}
A partial  type $\pi(\bar x)$  over $\emptyset$ is {\em amenable} if there is an $\aut(\C)$-invariant, Borel probability measure on $S_\pi(\C)$.
\end{dfn}

Let $\mu$ be a measure as in Definition \ref{definition: amenability of types}. Recall that the restriction of $\mu$ to the Baire sets is regular \cite[Theorem 7.1.5]{Dud}. Next, this restriction extends to a unique regular Borel probability measure $\nu$ (e.g. see \cite[Theorem 7.3.1]{Dud}). By the construction in the proof of \cite[Theorem 7.3.1]{Dud} and $\aut(\C)$-invariance of $\mu$, we get that $\nu$ is $\aut(\C)$-invariant. Thus, in Definition \ref{definition: amenability of types}, we can equivalently require a witnessing measure to be {\em regular} which we usually do.

\begin{rem}\label{remark: equivalent definitions of amenability of a type}
The following conditions are equivalent for a type $\pi(\bar x)$ over $\emptyset$. 
%(in items (3) and (4), assume that $\bar x$ is shorter than the degree of saturation of $\C$).
\begin{enumerate}
\item $\pi(\bar x)$ is amenable.
\item There is an $\aut(\C)$-invariant, Borel (regular) probability measure $\mu$ on $S_{\bar x}(\C)$ concentrated on $S_\pi(\C)$, i.e. for any formula $\varphi(\bar x, \bar a)$ inconsistent with $\pi(\bar x)$, $\mu([\varphi(\bar x,\bar a)])=0$ (where $[\varphi(\bar x,\bar a)]$ is the subset of $S_{\bar x}(\C)$ consisting of all types containing $\varphi(\bar x, \bar a)$).
\item %There is an $\aut(\C)$-invariant, finitely additive probability measure on relatively definable subsets of $\pi(\C)$.
There is an $\aut(\C)$-invariant, finitely additive probability measure on relatively $\C$-definable subsets of $\pi(\bar x)$.
\item %There is an $\aut(\C)$-invariant, finitely additive probability measure on definable subsets of $\C^{|\bar x|}$ concentrated $\pi(\C)$ (i.e. for any formula $\varphi(\bar x, \bar a)$ inconsistent with $\pi(\bar x)$, $\mu(\varphi(\C,\bar a))=0$).
There is an $\aut(\C)$-invariant, finitely additive probability measure on $\C$-definable sets in variables $\bar x$, concentrated on $\pi(\bar x)$ (i.e. for any formula $\varphi(\bar x, \bar a)$ inconsistent with $\pi(\bar x)$, $\mu(\varphi(\bar x,\bar a))=0$).
\end{enumerate}
\end{rem}

\begin{proof}
Follows easily using the fact (see \cite[Proposition 416Q(a)]{Fre} or \cite[Chapter 7.1]{Si}) that whenever $G$ acts by homeomorphisms on a compact, Hausdorff, 0-dimensional space $X$, then each $G$-invariant, finitely additive probability measure on the Boolean algebra of clopen subsets of $X$ extends to a $G$-invariant, Borel (regular) probability measure on $X$.
\end{proof}

%In fact, in items (3) and (4), we can also work with an unbounded tuple $\bar x$, computing sets of realizations in a bigger monster model. We can just think that a measure is defined on relatively definable subsets of $\pi(\bar x)$ or that it is concentrated on $\pi(\bar x)$.

Thus, by a {\em global $\aut(\C)$-invariant Keisler measure extending $\pi(\bar x)$} we mean a measure from any of the items of Remark \ref{remark: equivalent definitions of amenability of a type}. And similarly working over any model $M$ in place of $\C$. 

%%%Krzys Dec 2020: I added the next sentence, because the referee did not understand the notation $\mu(\bar x)$. In particular, I changed all $\mu (\bar x)$ to $\mu_{\bar x}$. But in most places, we really write just $\mu$ (without $\bar x$).
In order to emphasize that a Keisler measure $\mu$ is defined on a type space in variables $\bar x$, sometimes we will write $\mu_{\bar x}$.

\begin{prop}\label{proposition: absoluteness of amenability of a type}
Amenability of a given type $\pi(\bar x)$ (over $\emptyset$) is absolute in the sense that it does not depend on the choice of the monster model $\C$. It is also equivalent to the amenability of $\pi(\bar x)$ computed with respect to an $\aleph_0$-saturated and strongly $\aleph_0$-homogeneous model $M$ in place of $\C$.
\end{prop}

\begin{proof}
Let $M$ and $M'$ be two $\aleph_0$-saturated and strongly $\aleph_0$-homogeneous models.
Assume that there is an $\aut(M)$-invariant, Borel (regular) probability measure $\mu$ on $S_\pi(M)$. 
%%%Krzys Dec 2020: I added ``an'' before $Aut(M')$.
We want to find such an $\aut(M')$-invariant measure $\mu'$ on $S_\pi(M')$. 

Consider any formula $\varphi(\bar x, \bar a')$ with $\bar a' \in M'$. Choose (using the $\aleph_0$-saturation of $M$) any $\bar a \in M$ such that $\bar a' \equiv \bar a$, and define
$$\mu'([\varphi(\bar x, \bar a')] \cap S_{\pi}(M')):= \mu([\varphi(\bar x, \bar a)] \cap S_{\pi}(M)).$$ 
By the strong $\aleph_0$-homogeneity of $M$ and $\aut(M)$-invariance of $\mu$, we see that $\mu'$ is well-defined and $\aut(M')$-invariant. It is also clear that $\mu'(S_\pi(M'))=1$. It remains to check $\mu'$ is finitely additive on clopen subsets (as then $\mu'$ extends to the desired Borel measure). Take $\varphi(\bar x, \bar a')$ and $\psi(\bar x, \bar a')$ such that $[\varphi(\bar x, \bar a')] \cap S_{\pi}(M')$ is disjoint from $[\psi(\bar x, \bar a')] \cap S_{\pi}(M')$. This just means that $\varphi(\bar x, \bar a') \wedge \psi(\bar x, \bar a')$ is inconsistent with $\pi(\bar x)$. Take $\bar a \in M$ such that $\bar a \equiv \bar a'$. Then  $\varphi(\bar x, \bar a) \wedge \psi(\bar x, \bar a)$ is still inconsistent with $\pi(\bar x)$, so  
%%%Krzys Dec 2020: Since the computation below was unreadable for the referee, I shortened  it.
%$\mu'(([\varphi(\bar x, \bar a')] \cap S_{\pi}(M')) \cup ([\psi(\bar x, \bar a')] \cap S_{\pi}(M'))) = \mu(([\varphi(\bar x, \bar a)] \cap S_{\pi}(M)) \cup ([\psi(\bar x, \bar a)] \cap S_{\pi}(M))) = \mu([\varphi(\bar x, \bar a)] \cap S_{\pi}(M)) + \mu ([\psi(\bar x, \bar a)] \cap S_{\pi}(M)) = \mu'([\varphi(\bar x, \bar a')] \cap S_{\pi}(M')) + \mu'([\varphi(\bar x, \bar a')] \cap S_{\pi}(M'))$.
the obvious computation using additivity of $\mu$ yields: $\mu'(([\varphi(\bar x, \bar a')] \cap S_{\pi}(M')) \cup ([\psi(\bar x, \bar a')] \cap S_{\pi}(M'))) = \mu'([\varphi(\bar x, \bar a')] \cap S_{\pi}(M')) + \mu'([\psi(\bar x, \bar a')] \cap S_{\pi}(M'))$.
\end{proof}

%%%Krzys Dec 2020: I changed ``Remark'' to ``Proposition'', partly because of the referee's suggestion.
\begin{prop} \label{remark:sufficiency of countable, homogeneous models}  Assume $T$ to be countable, and let $\pi(\bar x)$ be a partial type. Then $\pi(\bar x)$ is amenable if and only if for all [sufficiently large] countable, ($\aleph_{0}$-)homogeneous models $M$, $\pi(\bar x)$ has an extension to a Keisler measure $\mu_{\bar x}$ over $M$ which is $\aut(M)$-invariant. If $T$ is uncountable,
the same is true but with ``countable, $\aleph_0$-homogeneous models'' replaced by ``strongly $\aleph_0$-homogeneous models of cardinality at most $|T|$''.
\end{prop}

Before we prove it, let us explain a few terms from the formulation. By {\em sufficiently large} models we mean the models from some class $\mathcal{M}$ of models which is closed under isomorphisms and such that for every finitary type $p \in S_n(\emptyset)$ there is a model $M \in \mathcal{M}$ with $p(M) \ne \emptyset$. A model $M$ is said to be {\em homogeneous} if for every finite tuples $\bar a \equiv \bar b$ from $M$ and $c \in M$ there is $d \in M$ with $\bar a c \equiv \bar b d$. If $M$ is countable, this is equivalent to strong $\aleph_0$-homogeneity. For any $A \subset \C$ of cardinality $\leq |T|$, a standard back-and-forth construction produces a strongly $\aleph_0$-homogeneous model $N$ of cardinality $\leq |T|$ and containing $A$.

\begin{proof}
%%%Krzys Dec 2020: Here, I wrote $\mu_{\bar M}$ instead of $\mu_{\bar M}(\bar x)$, but I added $\bar x$ in $\pi(\bar x)$. I also wrote ``in variables \bar x'' and added ``(with the index set ordered by inclusion)'' answering a question of the referee.
For each [sufficiently large] countable homogeneous model $M\prec \C$, let $\mu_{M}$ be an $\aut(M)$-invariant Keisler measure over $M$ extending $\pi(\bar x)$, and let ${\bar\mu}_M$ be an
arbitrary global Keisler measure extending $\mu_{M}$. Working in the compact space of global Keisler measures in variables $\bar x$, there is a subnet of the net $\{{\bar\mu}_M\}_M$ (with the index set ordered by inclusion), which converges to some $\bar\mu$. But then $\bar\mu$ is $\aut(\C)$-invariant:  Otherwise, for some formula $\phi({\bar x},{\bar y})$ and finite tuples $\bar a, \bar b$ in $\C$ with the same type, we have 
${\bar\mu}(\phi({\bar x},\bar a)) = r$ and ${\bar\mu}(\phi({\bar x},\bar b)) = s$ for some $r< s$. But then we can find some countable homogeneous model $M$ containing $\bar a,\bar b$ and such that 
${\bar\mu}_{M}(\phi({\bar x},\bar a)) < {\bar\mu}_M(\phi({\bar x},\bar b))$, contradicting the $\aut(M)$-invariance of $\mu_{M}$. 
\end{proof}

\begin{lem}\label{lemma: first lemma on amenable theories}
A type $\pi(\bar x)$ (over $\emptyset$) is amenable if and only if each formula $\varphi(\bar x)$ (without parameters) implied by $\pi(\bar x)$ is amenable. 
\end{lem}

\begin{proof}
The implication $(\rightarrow)$ is obvious, as $S_\pi(\C) \subseteq S_\varphi(\C)$, and so for any formula $\psi(\bar x, \bar a)$ we can define $\mu'([\psi(\bar x,\bar a)]\cap S_{\varphi}(\C)):= \mu([\psi(\bar x,\bar a)]\cap S_{\pi}(\C))$, where $\mu$ is an $\aut(\C)$-invariant, Borel probability  measure on $S_\pi(\C)$.

$(\leftarrow).$ Let $\Pi$ be the set of formulas in variables $\bar x$ which are implied by $\pi(\bar x)$. By assumption, for every formula $\varphi(\bar x) \in \Pi$ we can choose an $\aut(\C)$-invariant Keisler measure $\mu_{\varphi}$ which is concentrated on $\varphi(\bar x)$. In the compact space of all Keisler measures in variables $\bar x$, there is a subnet of the net $(\mu_{\varphi})_{\varphi \in \Pi}$ (with $\Pi$ ordered by implication of formulas) which converges to some Keisler measure $\mu$. Since all $\mu_{\varphi}$ are $\aut(\C)$-invariant, so is $\mu$. In order to see that $\mu$ is concentrated on $\pi(\bar x)$, consider any formula $\psi(\bar x,\bar a)$ inconsistent with $\pi(\bar x)$. Then there is a formula $\varphi(\bar x) \in \Pi$ inconsistent with $\psi(\bar x, \bar a)$. Then every $\theta \in \Pi$ which is implied by $\varphi(\bar x)$ is also inconsistent with $\psi(\bar x, \bar a)$, whence  $\mu_{\theta}(\psi(\bar x, \bar a))=0$. Thus, $\mu(\psi(\bar x, \bar a))=0$.
\end{proof}

%%%Krzys Dec 2020: I added ``(i.e. in finitely many variables)'', although now it is defined at the end of Section 1. The referee said that the statement of the next lemma is unclear. But for me it is very clear, so I did not change it (except recalling what ``finitary'' means).
\begin{lem}\label{lemma: second lemma on amenable theories}
All types in $S(\emptyset)$ (possibly in unboundedly many variables) are amenable if and only if all finitary (i.e. in finitely many variables) types in $S(\emptyset)$ are amenable.
\end{lem}

\begin{proof}
The implication $(\rightarrow)$ is trivial. For the other implication, take $p(\bar x) \in S_{\bar x}(\emptyset)$.
Consider the compact space $X:= [0,1]^{\{\varphi(\bar x,\bar a): \;\varphi(\bar x, \bar y)\; \textrm{a formula}, \;\bar a \in \C\}}$ with the pointwise convergence topology (where $\bar x$ is the fixed tuple of variables). Then the $\aut(\C)$-invariant, finitely additive probability measures on $\C$-definable sets in variables $\bar x$ concentrated on $p(\bar x)$ form a closed subset $\mathcal{M}$ of $X$. We can present $\mathcal{M}$ as the intersection of a directed family of closed subsets of $X$ each of which witnessing a finite portion of information of being in $\mathcal{M}$. But each such finite portion of information involves only finitely many variables, so the corresponding closed set is nonempty by the assumption that all finitary types are amenable and Remark \ref{remark: equivalent definitions of amenability of a type}. By the compactness of $X$, we conclude that $\mathcal{M}$ is nonempty.
\end{proof}

\begin{cor}\label{corollary: equivalent definitions of amenable theory}
The following conditions are equivalent.
\begin{enumerate}
\item All partial types (possibly in unboundedly many variables) over $\emptyset$ are amenable.
\item All complete types (possibly in unboundedly many variables) over $\emptyset$ are amenable.
\item All finitary complete types over $\emptyset$ are amenable.
\item All consistent formulas (in finitely many variables $\bar x$) over $\emptyset$ are amenable.
%%%Krzys Dec 2020: I added ``(where recall that $\bar c$ is an enumeration of $\C$)'', because the referee already forgot that $\bar c$ is an enumeration of $\C$.
\item $\tp(\bar c/\emptyset)$ is amenable (where recall that $\bar c$ is an enumeration of $\C$). 
\item $\tp(\bar m/\emptyset)$ is amenable for some tuple $\bar m$ enumerating a model.
\end{enumerate}
\end{cor}

\begin{proof}
The equivalence $(1) \leftrightarrow (2)$ is obvious (for $(2) \to (1)$ use the argument as in the proof of $(\rightarrow)$ in Lemma  \ref{lemma: first lemma on amenable theories}).
The equivalence $(2) \leftrightarrow (3)$ is Lemma \ref{lemma: second lemma on amenable theories}. The equivalence $(3) \leftrightarrow (4)$ follows from Lemma \ref{lemma: first lemma on amenable theories}. The implications $(1) \to (5) \to (6)$ are trivial. 
%Finally, $(6) \to (4)$ also follows from Lemma \ref{lemma: first lemma on amenable theories}, because taking all possible finite subtuples $\bar m'$ of $\bar m$ and $\varphi(\bar x') \in \tp (\bar m'/\emptyset)$, we will get all consistent formulas over $\emptyset$.
Let us show $(6) \to (4)$. Extend $\bar m$ to a tuple $\bar n$ also consisting of the elements of $M$ but so that each element of $M$ is repeated infinitely many times. Since the restriction map from $S_{\bar n}(\C)$ to $S_{\bar m}(\C)$ is an isomorphism of $\aut(\C)$-flows, by (6), we get that $\tp(\bar n/\emptyset)$ is amenable. Hence, $\tp(\bar n'/ \emptyset)$ is also amenable for every subtuple $\bar n'$ of $\bar n$. Thus, using Lemma \ref{lemma: first lemma on amenable theories}, we obtain (4),  because taking all possible finite subtuples $\bar n'$ of $\bar n$ and $\varphi(\bar x') \in \tp (\bar n'/\emptyset)$, we will get all consistent formulas over $\emptyset$ (up to permutations of variables).
\end{proof}

\begin{dfn}
The theory $T$ is {\em amenable} if the equivalent conditions of Corollary \ref{corollary: equivalent definitions of amenable theory} hold.
\end{dfn}

By Proposition \ref{proposition: absoluteness of amenability of a type}, we see that amenability of $T$ is really a property of $T$, i.e. it does not depend on the choice of $\C$.

Analogously, one can define the stronger notion of an extremely amenable theory.

\begin{dfn}
A type $\pi(\bar x)$  over $\emptyset$ is {\em extremely amenable} if there is an $\aut(\C)$-invariant type in $S_\pi(\C)$. The theory $T$ is {\em extremely amenable} if every type (in any number of variables) in $S(\emptyset)$ is extremely amenable. 
\end{dfn}

As in the case of amenability, compactness arguments easily show that the notions of extremely amenable types and extremely amenable theories are both absolute (i.e. do not depend on the choice of $\C$), and, in fact, they can be tested on any $\aleph_0$-saturated and strongly $\aleph_0$-homogeneous model in place of $\C$; moreover, $T$ is extremely amenable if and only if all finitary types in $S(\emptyset)$ are extremely amenable.  
%%%Krzys Dec 2020: I changed ``Remark'' to ``Proposition'' and removed ``note that'' in the second sentence below.
Note that Proposition \ref{remark:sufficiency of countable, homogeneous models} specializes to extremely amenable partial types, too.  So for countable theories, both amenability and extreme amenability can be seen at the level of countable models. It is also easy to see that in a stable theory, a type in $S(\emptyset)$ is extremely amenable if and only if it is stationary.\\

Yet another equivalent approach to amenability of $T$ is via $\aut(\C)$-invariant, finitely additive probability measures on the algebra of  so-called relatively definable subsets of $\aut(\C)$. This will be the exact analogue of the definition of definable amenability of definable groups (via the existence of an invariant Keisler  measure). We will use this approach in Section \ref{Subsection 4.2}.

%As usual, $\C$ is a monster model of an arbitrary complete theory $T$. Let $\bar c$ be an enumeration of $\C$ and let $S_{\bar c}(\C):=\{ \tp(\bar a/\C) \in S(\C): \bar a \equiv \bar c\}$. 
The idea of identifying $\aut(\C)$ with the subset $\{\sigma (\bar c): \sigma \in \aut(\C)\}$ of $\C^{\bar c}$ and considering relatively definable subsets of $\aut(\C)$, i.e. subsets of the form $\{ \sigma \in \aut(\C): \C \models \varphi(\sigma(\bar c),\bar c)\}$ for a formula $\varphi(\bar x, \bar c)$, already appeared in \cite[Appendix A]{KrPiRz}. Here, we extend this notion of relative definability to the  local context and introduce an associated notion of amenability which is easily seen to be equivalent to the amenability of $T$ [or of a certain type in the extended local version]. 
%we compare with other notions of amenability and with boundedness of minimal left ideals of the Ellis semigroup of the flow $(\aut(\C),S_{\bar c}(\C))$. 

%The idea is that, similarly to the context of a definable group $G$ in which definable amenability of $G$ extends the classical notion of amenability of $G$ as a discrete group, relative definable amenability of the automorphism group of a given model extends the classical notion of amenability of this group treated as a topological group with the pointwise convergence topology.

Let $M$ be any model of $T$ and let $\bar m$ be its enumeration.

\begin{dfn}\label{definition: relatively definable sets}
i) By a {\em relatively definable subset} of $\aut(M)$ we mean a subset of the form $\{ \sigma \in \aut(M) : M \models \varphi(\sigma(\bar m), \bar m)\}$, where $\varphi(\bar x, \bar y)$ is a formula without parameters.\\
ii) If $\bar \alpha$ is a tuple of some  elements of $M$, by {\em relatively $\bar \alpha$-definable subset} of $\aut(M)$ we mean a subset of the form $\{ \sigma \in \aut(M) : M \models \varphi(\sigma(\bar \alpha), \bar m)\}$, where $\varphi(\bar x, \bar y)$ is a formula without parameters.
\end{dfn}

The above definition differs from  the standard terminology in which ``$A$-definable'' means ``definable over $A$''; here, ``relatively $\bar \alpha$-definable'' has nothing to do with the parameters over which the set is relatively definable. One should keep this  in mind from now on.

For a formula $\varphi(\bar x,\bar y)$ and tuples $\bar a$, $\bar b$ from $M$ corresponding to $\bar x$ and $\bar y$, respectively, we will use the following notation 
$$A_{\varphi,\bar a, \bar b} = \{ \sigma \in \aut(M) : M \models \varphi(\sigma(\bar a),\bar b) \}.$$ 
When $\bar x$ and $\bar y$ are of the same length (by which we mean that they are also of the same sorts) and $\bar a= \bar b$, then this set will be denoted by $A_{\varphi, \bar a}$. 

%%%Krzys Dec 2020: I added the next sentence.
Note that for any tuple $\bar \alpha$ in $M$, the relatively $\bar \alpha$-definable subsets of $\aut(M)$ form a Boolean $\aut(M)$-algebra (i.e. a Boolean algebra closed under the action of $\aut(M)$ by left translations).

\begin{dfn}
i) The group $\aut(M)$ is said to be {\em relatively amenable} if there exists a left $\aut(M)$-invariant, finitely additive probability measure on the Boolean algebra of relatively definable subsets of $\aut(M)$.\\
ii)  If $\bar \alpha$ is a tuple of some elements of $M$, the group $\aut(M)$ is said to be {\em  $\bar \alpha$-relatively amenable} if there exists a left $\aut(M)$-invariant, finitely additive probability measure on the Boolean algebra of relatively $\bar \alpha$-definable subsets of $\aut(M)$. 
\end{dfn}

In particular, $\aut(M)$ being relatively amenable means exactly that it is $\bar m$-relatively amenable, where $\bar m$ is an enumeration of $M$.

We will mostly focus on the case when $M=\C$ is a monster model.
%or when $M$ is the unique countable model of an $\omega$-categorical theory. 
But often one can work in the more general context when $M$ is $\aleph_0$-saturated and strongly $\aleph_0$-homogeneous, including the case of the unique countable model of an $\omega$-categorical theory.

%%%Krzys Dec 2020: I expanded the the next remark by item (1). I also called this easy observation a proposition instead of remark as suggested by the referee. 
\begin{prop}\label{remark: relative definable amenability}
Let $M$ be $\aleph_0$-saturated and strongly $\aleph_0$-homogeneous enumerated as $\bar m$. Let $\bar \alpha$ be a tuple of some elements of $M$. Then we have:
\begin{enumerate}
\item The Boolean $\aut(M)$-algebra of clopen subsets of $S_{\bar \alpha}(M)$ is isomorphic to the Boolean $\aut(M)$-algebra of relatively $\bar \alpha$-definable subsets of $\aut(M)$.
\item The group $\aut(M)$ is $\bar \alpha$-relatively amenable if and only if there is an $\aut(M)$-invariant, (regular) Borel probability measure on $S_{\bar \alpha}(M)$ (equivalently, $\tp(\bar \alpha/\emptyset)$ is amenable). 
%%%Krzys Dec 2020: ``relatively definably amenable'' instead of ``relatively definable''.
In particular, $\aut(M)$ is relatively amenable if and only if there is an $\aut(M)$-invariant, (regular) Borel probability measure on $S_{\bar m}(M)$ (equivalently, $T$ is amenable).
\end{enumerate}
\end{prop}

%%%Krzys Dec 2020: The proof below is adjusted to the new formulation of the remark, and shortened.
\begin{proof}

(1) The assignment $[\varphi(\bar x, \bar m)] \mapsto A_{\varphi,\bar \alpha,\bar m}$ clearly defines a homomorphism between the two Boolean $\aut(M)$-algebras in question (and this does require any assumptions on $M$). The fact that it is an isomorphism follows easily from  $\aleph_0$-saturation and strongly $\aleph_0$-homogeneity of $M$.

(2) By (1), $\aut(M)$ is $\bar \alpha$-relatively amenable if and only if  there is an $\aut(M)$-invariant, finitely additive probability measure on the algebra of clopen subsets of $S_{\bar \alpha}(M)$ which in turn is equivalent to the existence of an $\aut(M)$-invariant, (regular) Borel probability measure on $S_{\bar \alpha}(M)$. The fact that the existence of an $\aut(M)$-invariant, (regular) Borel probability measure on $S_{\bar \alpha}(M)$ is equivalent to amenability of $\tp(\bar \alpha/\emptyset)$ follows from Proposition \ref{proposition: absoluteness of amenability of a type}. And then, the fact that the existence of an $\aut(M)$-invariant, (regular) Borel probability measure on $S_{\bar m}(M)$ is equivalent to amenability of $T$ follows from Corollary \ref{corollary: equivalent definitions of amenable theory}.
\end{proof}

So the terminologies ``$\aut(M)$ is [$\bar \alpha$-]relatively amenable'' and ``$T$ [resp. $\tp(\bar \alpha/\emptyset)$] is amenable'' will be used interchangeably.

\begin{cor}[]
[For a given tuple $\bar \alpha$, $\bar \alpha$-]relative amenability of $\aut(M)$ for an $\aleph_0$-saturated and strongly $\aleph_0$-homogeneous model $M$ [containing $\bar \alpha$] does not depend on the choice of $M$.
\end{cor}

%%%Krzys Dec 2020: In place of the old Corollary 2.15, I formulated below a definition, a comment about it, and Corollary 2.15. It has the same content as before, but separated (as the referee did not like the formulation of the old. 2.15).

The next corollary of Proposition \ref{remark: relative definable amenability} will play an essential role Section \ref{Subsection 4.2}. To state it, we need to extend Definition \ref{definition: relatively definable sets} as follows.

\begin{dfn}\label{definition: relatively type-definable sets}
i) If $\bar \alpha$ is a tuple of elements of $\C$, by a {\em relatively $\bar \alpha$-type-definable} subset of $\aut(\C)$, we mean a subset of the form  $\{ \sigma \in \aut(\C) : \C \models \pi(\sigma(\bar a), \bar b))\}$ for some partial type $\pi(\bar x, \bar y)$ (without parameters), where $\bar x$ and $\bar y$ are short tuples of variables, and $\bar a$, $\bar b$ are tuples from $\C$ corresponding to $\bar x$ and $\bar y$, respectively, such that $\bar a$ is a subtuple of $\bar \alpha$.\\
ii) By a  {\em relatively type-definable} subset of $\aut(\C)$, we mean a relatively $\bar c$-type definable subset; equivalently,  a subset  of $\aut(\C)$ of the form $\{ \sigma \in \aut(\C) : \C \models \pi(\sigma(\bar a), \bar b))\}$ for some partial type $\pi(\bar x, \bar y)$ (without parameters), where $\bar x$ and $\bar y$ are short tuples of variables, and $\bar a$, $\bar b$ are corresponding tuples from $\C$.
\end{dfn}

We will be using the Boolean algebra generated by all relatively  $\bar \alpha$-type-definable subsets of $\aut(\C)$. Observe that this algebra consists of all sets of the form  $\{\sigma \in \aut(\C): \tp(\sigma(\bar a)/A)) \in {\mathcal P}\}$, where $A\subseteq \C$ is a (small) set, $\bar a$ is a short subtuple of $\bar \alpha$, and ${\mathcal P}$ is a finite Boolean combination of closed subsets of $S_{\bar a}(A)$.

\begin{cor}\label{corollary: measure extends to relatively type-definable sets}
Assume $\tilde \mu$ is an $\aut(\C)$-invariant, (regular) Borel probability measure on $S_{\bar \alpha}(\C)$. For a set $X:=  \{\sigma \in \aut(\C): \tp(\sigma(\bar a)/A)) \in {\mathcal P}\}$ (where $A\subseteq \C$ is a (small) set, $\bar a$ is a short subtuple of $\bar \alpha$, and ${\mathcal P}$ is a finite Boolean combination of closed subsets of $S_{\bar a}(A)$), put $\mu(X) := \tilde{\mu}(\pi^{-1}[{\mathcal P}])$, where $\pi \colon S_{ \bar \alpha}(\C) \to S_{\bar a}(A)$ is the restriction map. Then $\mu$ is a well-defined, $\aut(\C)$-invariant, finitely additive probability measure on the Boolean algebra generated by relatively $\bar \alpha$-type-definable subsets of $\aut(\C)$.

In particular, if $\aut(\C)$ is relatively amenable, then there exists an $\aut(\C)$-invariant, finitely additive probability measure on the Boolean algebra generated by relatively type-definable subsets of $\aut(\C)$.

%If $\aut(\C)$ is relatively $\bar \alpha$-definably amenable, where $\bar \alpha$ is a tuple in $\C$ (e.g. $\bar \alpha =\bar c$), then there exists an $\aut(\C)$-invariant, finitely additive probability measure on the Boolean algebra generated by relatively $\bar \alpha$-type-definable sets, i.e. sets of the form  $\{ \sigma \in \aut(\C) : \C \models \pi(\sigma(\bar a), \bar b))\}$ for some partial type $\pi(\bar x, \bar y)$, where $\bar x$ and $\bar y$ are short tuples of variables, and $\bar a$, $\bar b$ are tuples from $\C$ corresponding to $\bar x$ and $\bar y$, respectively, such that $\bar a$ is a subtuple of $\bar \alpha$. 

%In particular, if $\aut(\C)$ is relatively definably amenable, then there exists an $\aut(\C)$-invariant, finitely additive probability measure on the Boolean algebra generated by relatively type-definable sets (i.e., relatively $\bar c$-type-definable sets).
\end{cor}

\begin{proof}
Easy exercise.
\end{proof}

%%%Krzys Dec 2020: I added the next sentence.
In the rest of this subsection, we give many examples (or even classes of examples) of amenable and extremely amenable theories.\\

%%%Krzys: I added the next paragraph.
Recall that a {\em $G$-flow} (for a topological group $G$) is a pair $(G,X)$, where $X$ is a compact, Hausdorff space on which $G$ acts continuously; a {\em $G$-ambit} is a $G$-flow $(G,X, x_0)$ with a distinguished point $x_0 \in  X$ with dense $G$-orbit (e.g. see p. 117-118 of \cite{Aus} for a discussion on universal ambits). The topological group $G$ is said to be {\em [extremely] amenable} if each $G$-flow (equivalently, the universal $G$-ambit) has an invariant, Borel probability measure [respectively, a fixed point]. (See \cite[Chapter III, Theorem 3.1]{Gla} for several equivalent definitions of amenability for topological groups.)

\begin{cor}\label{corollary: amenability implies relative definable amenability}
Let $M$ be $\aleph_0$-saturated and strongly $\aleph_0$-homogeneous. Then, if $\aut(M)$ is amenable as a topological group (with the pointwise convergence topology), then it is relatively amenable, which in turn implies that it is $\bar \alpha$-relatively amenable for any tuple $\bar \alpha$ of elements $M$.
%i) If $\aut(\C)$ is amenable as a topological group (with the pointwise convergence topology), then it is relatively definably amenable, which in turn implies that it is relatively $\bar \alpha$-definably amenable for any tuple $\bar \alpha$ of elements $\C$.\\
%ii) Assume $M$ is a countable, $\omega$-categorical structure. Then, if $\aut(M)$ is amenable as a topological group (with the pointwise convergence topology), then it is relatively definably amenable, which in turn implies that it is relatively $\bar \alpha$-definably amenable for any tuple $\bar \alpha$ of elements $M$.

Similarly, extreme amenability of $\aut(M)$ as a topological group implies extreme amenability of $T$.
\end{cor}

\begin{proof}
Amenability of $\aut(M)$ implies that there is an $\aut(M)$-invariant, Borel probability measure on $S_{\bar m}(M)$. By Proposition \ref{remark: relative definable amenability}, this implies relative amenability of $\aut(M)$. Furthermore, since there is an obvious flow homomorphism from $S_{\bar m}(M)$ to $S_{\bar \alpha}(M)$, a measure on $S_{\bar m}(M)$ induces a measure on $S_{\bar \alpha}(M)$, and this is enough by Proposition \ref{remark: relative definable amenability}.
%(i) Amenability of $\aut(\C)$ implies that there is a left $\aut(\C)$-invariant, Borel probability measure on $S_{\bar c}(\C)$. By Remark \ref{remark: relative definable amenability}, this implies relative definable amenability of $\aut(\C)$. Furthermore, since there is an obvious flow homomorphism from $S_{\bar c}(\C)$ to $S_{\bar \alpha}(\C)$, a measure on $S_{\bar c}(\C)$ induces a measure on $S_{\bar \alpha}(\C)$, and this is enough by Remark \ref{remark: relative definable amenability}.\\
%(ii) is similar.
\end{proof}

%%%Krzys Dec 2020: The referee had troubles to follow brackets in the text below, so I added ``resp.'' in a few places. I also removed one pair of parentheses. It seems that the referee has a problem to understand longer sentences, or longer pieces of text, etc.
As in the introduction, we will call a countable $\aleph_{0}$-categorical theory KPT-[extremely] amenable if the automorphism group of its unique countable model is [resp. extremely] amenable as a topological group.

%So, by  Corollary \ref{corollary: amenability implies relative definable amenability}, KPT-[extreme] amenability of a (countable $\aleph_{0}$-categorical) theory $T$ implies [extreme] amenability of $T$ in the new sense of this paper.  In fact most, if not all, of the examples of not only KPT-extremely amenable theories (such as dense linear orderings) but also KPT-amenable (not necessarily KPT-extremely amenable) theories (such as the random graph) come from Fra\"{i}ss\'{e} classes with {\em canonical amalgamation}  hence are extremely amenable in our sense.  Only canonical amalgamation over $\emptyset$ is needed here and says that there is a map $\otimes$ taking pairs of finite structures $(A,B)$to an amalgam $A\otimes B$, such that if $C$ is a substructure of $B$ then the identity map on $A$ and $C$ makes $A\otimes C$ a substructure of $A\otimes B$.  This easily yields, for $A$ a finite structure in the class, and $M$ the unique countable model of the the $\omega$-categorical theory determined by the Fra\"{i}ss\'{e} class, a structure $A\cup M$ such that any isomorphism between finite substructures of $M$ extends to an isomorphism of A=cup M with itself which fixes $A$ pointwise and is an automorphism of $M$.  By quantifier elimination this gives an  extension of $tp(A)$ to a type over $M$ invariant under $Aut(M)$  (which is enough to yield extreme amenability of $T$).   A typical example is a Fra\"{i}ss\'{e} class with ``free amalgamation", namely adding no new relations. 

So, by  Corollary \ref{corollary: amenability implies relative definable amenability}, KPT-[extreme] amenability of a countable, $\aleph_{0}$-categorical theory $T$ implies [resp. extreme] amenability of $T$ in the new sense of this paper.  In fact most, if not all, of the examples of not only KPT-extremely amenable theories (such as dense linear orderings) but also KPT-amenable 
%%%Krzys Dec 2020: I removed ``(not necessarily KPT-extremely amenable)'' to make it easier to read.
%(not necessarily KPT-extremely amenable)
theories (such as the random graph, \cite[p. 2062]{AKL}) come from Fra\"{i}ss\'{e} classes with {\em canonical amalgamation}, hence are extremely amenable in our sense.  Only canonical amalgamation over $\emptyset$ is needed here (see Proposition \ref{proposition: canonical amalgamation} below) which says that there is a map $\otimes$ taking pairs of finite structures $(A,B)$ from the Fra\"{i}ss\'{e} class to an amalgam $A\otimes B$ (also in the Fra\"{i}ss\'{e} class) which is compatible with embeddings, i.e. if $f \colon B \to C$ is an embedding of finite structures from the Fra\"{i}ss\'{e} class, then there exists an embedding from $A \otimes B$ to $A \otimes C$ which commutes with $f$ and with the embeddings: $A \to A \otimes B$, $B \to A \otimes B$, $A \to A \otimes C$, and $C \to A \otimes C$.  A typical example is a Fra\"{i}ss\'{e} class with ``free amalgamation'', namely adding no new relations. 

%Let us briefly explain why canonical amalgamation of a Fra\"{i}ss\'{e} class of finite structures in a relational language [or, more generally, finitely generated structures in any language] whose Fra\"{i}ss\'{e} limit $M$ is $\omega$-categorical implies extreme amenability. First, note that canonical amalgamation implies that for any finite tuples $\bar d, \bar a_1, \bar b_1,\dots, \bar a_n, \bar b_n$ from $M$, if the structures $\bar a_i$ and $\bar b_i$ are isomorphic (i.e. have the same quantifier-free type), then we can amalgamate structures $\bar d$ and $(\bar a_i, \bar b_i: i \leq n)$ into a structure $\bar d',\bar a_1',\bar b_1',\dots ,\bar a_n',\bar b_n'$ in such a way that $\bar a_i'$ is isomorphic with $\bar b_i'$ over $\bar d'$ for all $i \leq n$. Therefore, using $\omega$-categoricity and quantifier elimination, one concludes by compactness that any finitary type in $S(\emptyset)$ extends to an $\aut(M)$-invariant type in $S(M)$, so $T$ is extremely amenable (since $M$ is $\omega$-categorical). 

\begin{prop}\label{proposition: canonical amalgamation}
If $M$ is an $\omega$-categorical structure which is the Fra\"{i}ss\'{e} limit of a Fra\"{i}ss\'{e} class of finite structures in a relational language [or, more generally, finitely generated structures in any language] with canonical amalgamation over $\emptyset$, then $\Th(M)$ is extremely amenable.
\end{prop}

\begin{proof}
For simplicity we deal with the case of finite relational structures. All the structures below are from the Fra\"{i}ss\'{e} class in question.

\medskip

\begin{clm}
For any finite tuples $\bar d, \bar a_1, \bar b_1,\dots, \bar a_n, \bar b_n$ from $M$, if the structures $\bar a_i$ and $\bar b_i$ are isomorphic (i.e. have the same quantifier-free type), then we can amalgamate structures $\bar d$ and $(\bar a_i, \bar b_i: i \leq n)$ into a structure $\bar d',\bar a_1',\bar b_1',\dots ,\bar a_n',\bar b_n'$ in such a way that $\bar a_i'$ is isomorphic with $\bar b_i'$ over $\bar d'$ for all $i \leq n$.
\end{clm}

\begin{clmproof}
Let $A$ be the substructure of $M$ consisting of the coordinates of the tuple $\bar d$, $B$ the substructure of $M$ consisting of the coordinates of the tuples $\bar a_1,\dots, \bar a_n$, and $C$ a substructure of $M$ containing the coordinates of the tuples $\bar a_1,\dots, \bar a_n,\bar b_1,\dots, \bar b_n$ and such that for every $i$ the isomorphism $\bar a_i \mapsto \bar b_i$ extends to an embedding $\sigma_i \colon B \to C$. Let $D:=C \otimes B$.

%%%%%%%%%%%%%%%%%%%%%%%%%%%%%%%%%%%%%%%%%%%%%%%%%
\begin{comment}
By canonical amalgamation, we have the following diagram of embeddings:

	\begin{figure}[H]
		\centering
		\begin{tikzcd}
& A \otimes D &\\
& A \otimes B  \arrow[u, hook, "g_2"] & D=C \otimes B\arrow[lu,hook,"f_2"]\\
A \arrow[ruu, hook, "\delta"]\arrow[ru,hook,"\beta"] & B\arrow[u,hook,"g_1"]\arrow[ru,hook,"h"] & C\arrow[u,hook,"f_1"] 
		\end{tikzcd}
	\end{figure}

Let $f:=f_2 \circ f_1 \colon B \to A \otimes D$ and $g:=g_2 \circ g_1 \colon C \to A \otimes D$.
%Let $\sigma_i^+,\sigma_i^-: B \colon C$ be given by $\sigma_i^+(\bar a_i)=\bar a_i$ and $\sigma_i^-(\bar a_i)=\bar b_i$. 
Let $\tau_i:= f_1 \circ \sigma_i \colon B \to C$. Let $s \colon B \to D$ be the embedding given by $s(\bar a_i):=f_1(\bar a_i)$ for all $i$.
By canonical amalgamation, we also have the following diagrams of embeddings.
\end{comment}
%%%%%%%%%%%%%%%%%%%%%%%%%%%%%%%%%%%%%%%%%%%%%%%

We have the following collection of canonical embeddings:

	\begin{figure}[H]
		\centering
		\begin{tikzcd}
& A \otimes D &\\
& A \otimes B  & D=C \otimes B\arrow[lu,hook,"f_2"]\\
A \arrow[ruu, hook, "\delta"]\arrow[ru,hook,"\beta"] & B\arrow[u,hook,"g_1"]\arrow[ru,hook,"h"] & C\arrow[u,hook,"f_1"] 
		\end{tikzcd}
	\end{figure}

Let $f:=f_2 \circ f_1 \colon C \to A \otimes D$ and $\tau_i:= f_1 \circ \sigma_i \colon B \to D$. Let $s \colon B \to D$ be the embedding given by $s(\bar a_i):=f_1(\bar a_i)$ for all $i$.

By canonical amalgamation, we have the following commutative diagrams of embeddings:

	\begin{figure}[H]
  \tikzset{column sep=small, ampersand replacement=\&}
\begin{floatrow}
		\centering
		\ffigbox{\begin{tikzcd}
\& A\arrow[ld,hook,"\beta"]\arrow[rd,hook,"\delta"] \&\\
A \otimes B \arrow[rr,hook, "\Phi_i"] \& \&A \otimes D\\
B\arrow[u,hook,"g_1"]\arrow[rr,hook,"\tau_i"] \& \&D\arrow[u,hook,"f_2"]\\
		\end{tikzcd}}{}
%	\end{figure}
%
%	\begin{figure}[H]
%
%		\centering
		\ffigbox{\begin{tikzcd}
\& A\arrow[ld,hook,"\beta"]\arrow[rd,hook,"\delta"] \&\\
A \otimes B \arrow[rr,hook, "\Psi"] \& \&A \otimes D\\
B\arrow[u,hook,"g_1"]\arrow[rr,hook,"s"] \& \&D\arrow[u,hook,"f_2"]\\
		\end{tikzcd}}{}
\end{floatrow}
	\end{figure}

We claim that $\bar a_i':=f(\bar a_i)$, $\bar b_i':=f(\bar b_i)$, and $\bar d':=\delta(\bar d)$ are as required. It is clear that $(\bar a_i, \bar b_i: i \leq n)$ is isomorphic to $(\bar a_i, \bar b_i: i \leq n)$, and $\bar d$ to $\bar d'$. The fact that $\bar a_i' \equiv^{qf}_{\bar d'} \bar b_i'$ can be seen via the following computation based on the above diagrams: $\Phi_i (\Psi^{-1}(\bar a_i')) = \Phi_i(\Psi^{-1}(f_2(f_1(\bar a_i))))=\Phi_i(\Psi^{-1}(f_2(s(\bar a_i))))= \Phi_i(g_1(\bar a_i))=f_2(\tau_i(\bar a_i))=f_2(f_1(\sigma_i(\bar a_i)))= f(\bar b_i)=\bar b_i'$, and  $\Phi_i (\Psi^{-1}(\bar d'))=\Phi_i(\Psi^{-1}(\delta(\bar d))) = \Phi_i(\beta(\bar d))=\delta(\bar d)=\bar d'$.
\end{clmproof}

By the claim, using $\omega$-categoricity and quantifier elimination, one concludes by compactness that any finitary type in $S(\emptyset)$ extends to an $\aut(M)$-invariant type in $S(M)$, so $T$ is extremely amenable (since $M$ is $\omega$-categorical). 
\end{proof}

%%%Krzys Dec 2020: The referee asked for a more detailed analysis of the examples from the next paragraph, so I wrote it as an ``Example'' and gave a proof. Also, I gave a concrete example in 2.18 with a proof.
%In \cite{KrPi}, we proved that both KPT-amenability and KPT-extreme amenability are preserved by adding finitely many parameters.  This is not the case for our notion of first order [extreme] amenability.  For example, if $T$ is the theory of two equivalence relations $E_{1}, E_{2}$, where $E_{1}$ has infinitely many classes, all infinite, and each $E_{1}$-class is divided into two $E_{2}$-classes, both infinite, then $T$ is extremely amenable, but adding an (imaginary) parameter for an $E_{1}$-class destroys extreme amenability. Similar examples can be built by putting uniformly in each $E_{1}$-class some non amenable theory. 
In \cite{KrPi}, we proved that both KPT-amenability and KPT-extreme amenability are preserved by adding finitely many parameters.  This is not the case for our notion of first order [extreme] amenability as shown by the following two examples. Before that observe that if $T$ is [extreme] amenability, then so is $T^{eq}$. 

\begin{ex}\label{example: naming parameters}
Let $T$ be the theory of two equivalence relations $E_{1}, E_{2}$, where $E_{1}$ has infinitely many classes, all infinite, and each $E_{1}$-class is divided into two $E_{2}$-classes, both infinite. Then $T$ is extremely amenable, but adding an (imaginary) parameter for an $E_{1}$-class destroys extreme amenability.
\end{ex}

\begin{proof}
Using a standard back-and-forth argument, one checks that $T$ has quantifier elimination. To show extreme amenability of $T$, consider any type $\pi(\bar x)$ without parameters. It is clear (using quantifier elimination) that $\pi(\bar x)$ extends to $p \in S_{\bar x}(\C)$ so that $\neg E_1(x_i,c) \in p$ for all coordinates $x_i$ of $\bar x$ and all $c \in \C$. Then also $\neg E_2(x_i,c) \in p$ for all $x_i$ and $c$ as before. By quantifier elimination, $p$ is clearly $\aut(\C)$-invariant.

Now, add a constant for an $E_1$-class $C$. Consider the formula (over $\emptyset$) $\varphi(x) := (x \in C)$.  Let $p \in S_x(\C)$ be any global type containing $\varphi(x)$. Then $p$ determines one of the two $E_2$-classes into which $C$ is divided by $E_2$. But, by quantifier elimination, there is an automorphism of $\C$ which swaps these two classes, so moves $p$, and hence $p$ is not $\aut(\C)$-invariant.    
\end{proof}

\begin{ex}
Let $T$ be the theory of an equivalence relation $E$ with infinitely many infinite classes and a ternary relation $S(x,y,z)$ which is exactly the disjoint union of dense circular orders on all the $E$-classes. Then $T$ is extremely amenable, but adding an (imaginary) parameter for an $E$-class destroys even amenability.
\end{ex}

\begin{proof}
Again we have quantifier elimination, and extremely amenability can be seen as in the last example. Now, add a constant for an $E$-class $C$. Consider the formula (over $\emptyset$) $\varphi(x) := (x \in C)$. By quantifier elimination, the structure induced on $C$ is interdefinable over $\emptyset$ with a monster model $\C'$ of the theory of a dense circular order, and after dividing by the kernels $K$ and $K'$ of the relevant actions, the $\aut(\C)/K$-flow $S_{\varphi}(\C)$ is isomorphic to the $\aut(\C')/K'$-flow $S_1(\C')$. Since the latter flow does not carry an invariant, Borel probability measure (by Remark \ref{remark: amenability implies extension base} below, because the formula $x=x$ forks over $\emptyset$  in dense circular orders),  neither does the former one. 
\end{proof}

We take the opportunity and modify the above example to get a $G$-compact theory whose expansion by an imaginary parameter is not $G$-compact.

\begin{ex}\label{example: naming parameters kills G-compactness}
Let $T_0$ be any single-sorted non $G$-compact theory (e.g. from \cite[Proposition 4.5]{CLPZ})
in a language $L_0$ without constants.  
Let $T$ be the theory of an equivalence relation $E$ with infinitely many classes,
%and an $L_0$-structure which is the disjoint union of models of $T_0$ on all $E$-classes.
and each class has the $L_0$-structure of a model of $T_0$ (and for any $n$-ary function symbol  $f$ and a tuple $\bar a=(a_0,\dots,a_{n-1})$ containing elements from at least 2 different $E$-classes, we have $f(\bar a)=a_0$).
Then $T$ is $G$-compact (even $G$-trivial), but $G$-compactness is destroyed by adding an (imaginary) parameter for an $E$-class.
\end{ex}

\begin{proof}
By a back-and-forth argument, $T$ has quantifier elimination relative to the Morleyization of $T_0$.
Recall that  $\C\models T$ is a monster model. It is a disjoint union of monster models of $T_0$ on all $E$-classes.  Note that any union of infinitely many $E$-classes is a model of $T$, in fact, an elementary substructure of $\C$.  
%Also, any collection of automorphisms of all $L_0$-structures $[a]_E$, where $a$ ranges over $\C$, extends to an automorphism of $\C$. 
Using this observation, one easily gets that any two tuples of bounded length with the same type over $\emptyset$ have the same type over a model, i.e. they are at Lascar distance at most 1. Thus, $T$ is $G$-trivial. On the other hand, take any tuple $\bar a$ contained in a single $E$-class $C$ and such that $[\bar a]_{E_L} \ne [\bar a]_{E_{KP}}$ in the sense of $T_0$ (working in $C \models T_0$). Then $[\bar a]_{E_L} \ne [\bar a]_{E_{KP}}$ in the sense of $T$ expanded by a constant for the class $C$, which follows from the obvious observation that $\aut^{T_0}(C)$ coincides with $\aut(\C/(\C \setminus C))|_C$ and the result from \cite{Ne} saying that $[\bar a]_{E_L} \ne  [\bar a]_{E_{KP}}$ if and only if $[\bar a]_{E_L}$ has infinite Lascar diameter.
\end{proof}

Let us look at Example \ref{example: naming parameters kills G-compactness} from the perspective of the main result of this paper.
By an argument as in Example \ref{example: naming parameters}, the theory $T$ from Example \ref{example: naming parameters kills G-compactness} is extremely amenable. Thus, $G$-triviality alternatively follows from Proposition \ref{proposition: extremely amenable local}. On the other hand, since after adding a constant the theory is not $G$-compact, it is also nonamenable by Theorem \ref{The main theorem}.

\begin{ex}
%Let $M_n$ be a dense circular order $S_n$ equipped with the rotation $g_n$ of order $n$, and $T_n:=\Th(M_n)$; the language of $T_n$ consists of $S_n$ and $g_n$ and will be denoted by $L_n$. Let $T^-$ be the theory of pairwise disjoint unary predicates $P_n$, $n\in\omega$, where each $P_n$ is equipped with an equivalence relation $E_n$ with infinitely many classes and each $P_n$ is an $L_n$-structure which is the disjoint union of models of $T_n$ on all $E_n$-classes. Finally let $T$ be the theory $T$ expanded by the addition sort of $E_1
Let $M_n$ be the  unit circle equipped with the ternary relation $S_n$ of circular order and with the clockwise rotation $g_n$ of order $n$, and let $T_n:=\Th(M_n)$ (see \cite[Section 4]{CLPZ}); the language of $T_n$, consisting of $S_n$ and $g_n$, will be denoted by $L_n$. Let $T^-$ be the theory of an equivalence relation $E$ and pairwise disjoint unary predicates $P_n$, $n\in\omega$, where each $P_n$ is a union of infinitely many $E$-classes,
%and is equipped with an $L_n$-structure which is the disjoint union of models of $T_n$ on all $E$-classes on $P_n$. 
and each $E$-class on $P_n$ has the $L_n$-structure of a model of $T_n$ (and $g_n(a) =a$ for $a \notin P_n$).
Finally, let $T$ be the theory $T^-$ canonically expanded by the additional sort of all $E$-classes (but we do not add anything else from $T^{eq}$). Then $T$ is $G$-compact after naming any finitely many parameters, but is not $G$-compact after naming countably many elements $a_n/E$, $n<\omega$ (where $a_n \in P_n$ are chosen arbitrarily).
\end{ex}

\begin{proof}
By a back-and-forth argument, $T$ has quantifier elimination. For the monster model $\C$ of $T$ we have that each $E$-class on $P_n$ is a monster model of $T_n$. Again, any union of $E$-classes which contains infinitely many $E$-classes on each $P_n$ is an elementary substructure of $\C$.  
%Also, any collection of automorphisms of all $L_n$-structures $[a]_E$, where $a$ ranges over $P_n$, and $n$ ranges over $\omega$, extends to an automorphism of $\C$. 
Using this observation, the fact that $T$ is $G$-compact after naming any finitely many parameters follows from the observation that each theory $T_n$ has this property (which we leave as an exercise).

Now, consider any tuple $\bar a= (a_n)_{n \in \omega}$ with $a_n \in P_n$. Add constants for all the classes $a_n/E$, $n \in \omega$, and denote the resulting expansion of $T$ by $T^c$. Then $[a_n]_{E_L}$ in the sense of $T_n$ (working in $[a_n]_E \models T_n$) coincides with $[a_n]_{E_L}$ in the sense of $T^c$, and the Lascar diameters of $[a_n]_{E_L}$ in the sense of both theories also agree. By \cite[Corollary 4.4]{CLPZ}, this diameter is greater than $n/2$. Thus, the diameter of $[\bar a]_{E_L}$ (in $T^c$) is infinite, so $T^c$ is not $G$-compact.
\end{proof}

The next fact follows easily from \cite[Proposition 2.11]{HrPi}.
\begin{fct}\label{fact: charact. of non-forking in NIP theories}
 In an NIP theory, for any global type $p$ the following conditions are equivalent:
\begin{enumerate} 
\item $p$ does not fork over $\emptyset$.
\item The $\aut(\C)$-orbit of $p$ is bounded.
\item $p$ is Kim-Pillay invariant (i.e. invariant under $\autf_{KP}(\C)$).
\item $p$ is Lascar invariant.
\end{enumerate}
\end{fct}

More importantly, Proposition 4.7 of \cite{HrPi} can be stated as:
\begin{fct}\label{fact: amenable = non forking}
In an NIP theory, a type $p \in S(\emptyset)$ is amenable if and only if it does not fork over $\emptyset$ (equivalently, it has a global non-forking extension).
\end{fct}

Although Fact \ref{fact: amenable = non forking} is proved in \cite{HrPi}, let us give a sketch of the proof of ($\leftarrow$) in order to see how the desired measure is obtained. 
So assume that $p(\bar x)$ does not fork over $\emptyset$, and take its global non-forking extension $q(\bar x) \in S_{\bar x}(\C)$. Take $\bar \alpha \models p$.  Consider any formula $\varphi(\bar x, \bar b)$. Recall that
$$A_{\varphi,\bar \alpha, \bar b} =\{ \sigma \in \aut(\C): \C \models \varphi(\sigma(\bar \alpha),\bar b)\}.$$ 
Let $S_\varphi:=\{ \bar b': \varphi(\bar x, \bar b') \in q\}$.
By Fact \ref{fact: charact. of non-forking in NIP theories}, $q$ is $\autf_{KP}(\C)$-invariant. So, by the argument in Proposition 2.6 of \cite{HrPi} and NIP, there is $N<\omega$ such that
$$S_\varphi=\bigcup_{n<N} A_n \cap B_{n}^c,$$
where each $A_n$ and $B_n$ is type-definable and invariant under $\autf_{KP}(\C)$.
Let 
$$\tilde{S}_{\varphi(\bar x,\bar b)}:= \{ \sigma/\autf_{KP}(\C) : \varphi(\bar x, \sigma^{-1}(\bar b)) \in q\} =  \{ \sigma/\autf_{KP}(\C) : \varphi(\bar x, \bar b) \in \sigma (q)\}.$$
%and let $\pi \colon \aut(\C) \to \gal_{KP}(T)$ be the quotient map.
Using the above formula for $S_\varphi$, one shows that  $\tilde{S}_{\varphi(\bar x,\bar b)}$ is a Borel (even constructible) subset of $\gal_{KP}(T)$. 

Let $\frak{h}$ be the unique (left invariant) normalized Haar measure on the compact group $\gal_{KP}(T)$. By the last paragraph, $\tilde{S}_{\varphi(\bar x,\bar b)}$ is Borel, hence $\frak{h}(\tilde{S}_{\varphi(\bar x,\bar b)})$ is defined, and so we can put 
$$\mu (A_{\varphi,\bar \alpha, \bar b}) := \frak{h}(\tilde{S}_{\varphi(\bar x, \bar b)}).$$
It is easy to check that $\mu$ is a well-defined (i.e. does not depend on the choice of $\varphi$ yielding the fixed set $A=A_{\varphi,\bar \alpha, \bar b}$), $\aut(\C)$-invariant, finitely additive probability measure on relatively $\bar \alpha$-definable subsets of $\aut(\C)$. Thus, $\aut(\C)$ is $\bar \alpha$-relatively amenable; equivalently, $p$ is amenable.

%%%Krzys: The next sentence is removed.
%By Proposition \ref{proposition: relative definable amenability is equivalent to the existence of bounded orbit} and the discussion preceding it, we get the following corollary, yielding a large class of amenable theories.

%%%Krzys Dec 2020: I added the next sentence with a reference (the referee asked for a reference). Do you know a reference where the o-minimal case is proved (namely that $\emptyset$ is an extension base)?

By Fact \ref{fact: amenable = non forking} and \cite[Corollary 2.14]{HrPi}, we have

%%%Krzys: As Anand suggested, at the very end of the next proposition I added: "(even after adding constants)".
\begin{cor}\label{corollary:NIP characterization}
Assume $T$ has NIP. Then, $T$ is amenable if and only if  $\emptyset$ is an extension base (i.e. any type over $\emptyset$ does not fork over $\emptyset$). In particular, stable, o-minimal, and c-minimal theories are all amenable (even after adding constants). 
\end{cor}

%%%Krzys: The next paragraph, suggested by Anand, is new.
%%%Krzys Dec 2020: I wrote the next thing as a remark (as requested by the referee), and added a comment before it.
In the above proof of the implication ($\leftarrow$) in Fact \ref{fact: amenable = non forking}, NIP plays an essential role to get that $q$ is $\autf_{KP}(\C)$-invariant and that $\tilde{S}_\varphi$ is Borel. On the other hand, the implication ($\rightarrow$) in Fact \ref{fact: amenable = non forking} is completely general. Namely, we have

\begin{rem}\label{remark: amenability implies extension base}
In an arbitrary theory, if a partial type $\pi(\bar x)$ over $\emptyset$ is amenable, then it does not fork over $\emptyset$.
In particular, in an arbitrary amenable theory, $\emptyset$ is an extension base. 
\end{rem}
\begin{proof}
%Consider an arbitrary type $p(\bar x)  \in S_{\bar x}(\emptyset)$. 
Let $\mu$ be a global, invariant Keisler measure extending $\pi(\bar x)$. Choose a $\mu$-wide type $q(\bar x) \in S_{\bar x}(\C)$ (i.e. any formula in $q(\bar x)$ is of positive measure). Then, one easily checks that $q(\bar x)$ does not fork over $\emptyset$, so we are done. 
\end{proof}

Thus, amenability of $T$ is a strong form of saying that $\emptyset$ is an extension base; and amenability after adding any constants is a strong form of saying that every set is an extension base.

%Therefore, relative definable amenability of $\aut(\C)$ is absolute in NIP theories. This also yields a class of theories for which $\aut(\C)$ is relatively definably amenable, namely all theories with NIP for which $\emptyset$ is an extension base.

%
%By Proposition \ref{proposition: relative definable amenability is equivalent to the existence of bounded orbit}, we conclude that under NIP, $\aut(\C)$ is relatively definably amenable if and only if $\emptyset$ is an extension base, which implies immediately that relative definable amenability of $\aut(\C)$ is absolute in NIP theories. This also yields a class of theories for which $\aut(\C)$ is relatively definably amenable, namely all theories with NIP for which $\emptyset$ is an extension base.
 
By \cite[Corollary 2.10]{HrPi}, the characterization from Corollary \ref{corollary:NIP characterization} gives us

\begin{cor}
Assume $T$ has NIP. Then amenability of $T$ implies $G$-compactness. 
\end{cor}
%
%%%Krzys: I changed the next sentence. In particular, instead of referring to Section 4, I referred to Theorem 0.1.
%In Section \ref{Subsection 4.2}, we will generalize this corollary to arbitrary theories, using different methods.
Theorem \ref{The main theorem} is a generalization of the last corollary to arbitrary amenable theories, but it requires completely different methods compared with the NIP case.

%%%Krzys: I added "(and so to amenability of $T$)". 
It is worth mentioning that Theorem 7.7 of \cite{KrNeSi} yields several other conditions equivalent (under NIP) to the existence of $p \in S_{\bar c}(\C)$ with bounded $\aut(\C)$-orbit (and so to amenability of $T$), for example: some (equivalently, every) minimal left ideal of the Ellis semigroup of the $\aut(\C)$-flow $S_{\bar c}(\C)$ is of bounded size. 
%%%Krzys: I added the next two sentences.
In particular, a variant of Newelski's conjecture proved in \cite[Theorem 0.7]{KrNeSi} can be stated as follows: if $T$ is an amenable theory with NIP, then a certain natural epimorphism from the Ellis group of $T$ to $\gal_{KP}(T)$ is an isomorphism. This also implies $G$-compactness of amenable, NIP theories.

%%%Krzys Dec 2020: I removed ``finally'' in the next sentence, as below I added some material (with examples).
Let us mention in this section some relations between our notions of amenability and extreme amenability of a theory $T$ and the notion of a strongly determined over $\emptyset$ theory from \cite{Ivanov} (originating in work of Ivanov and Macpherson \cite{I-M}).  Decoding the definition in \cite{Ivanov}, $T$ is strongly determined over 
$\emptyset$ if any complete type $p(\bar x)$ over $\emptyset$ has an extension to a complete type $p'(\bar x)$ over $\C$ which is $\acl^{eq}(\emptyset)$-invariant.  So clearly $T$ extremely amenable implies $T$ is strongly determined over 
$\emptyset$.   Moreover, by Corollary \ref{corollary:NIP characterization}, assuming NIP, $T$ strongly determined over 
$\emptyset$ implies amenability of $T$.  In fact, if $T$ is NIP and KP-strong types agree with usual strong types (over $\emptyset$), then $T$ is strongly determined over $\emptyset$ iff $T$ is amenable. \\

%%%Krzys Dec 2020: The rest of this section is new. If you have more examples, please add them. The referee was complaining that there were very few examples where the results (i.e. our main theorem) applies and that it is difficult to say how it compares to the current state of knowledge. I added the list below, but in all (but the last two) of these examples the main theorem was already known (such as in the amenable NIP or in the supersimple context) or very easy (in an extremely amenable case). That is why I also added these last examples below (which refer to my later preprint with J. Lee and S. Moconja). If you have more interesting examples of amenable, but non-NIP, non-simple, and non-extremely amenable theories, it would be good to include them. 
%(Maybe, similarly to this last example below, one could add ``independently'' a random graph structure on the field sort in ACVF....).
%Let us summarize some examples of classes of [extremely] amenable theories. 
We finish this section with a list of some  examples of classes of [extremely] amenable theories.

\begin{itemize}
\item By Corollary \ref{corollary: amenability implies relative definable amenability}, every countable, $\omega$-categorical structure with [extremely] amenable group of automorphisms has [extremely] amenable theory. Many concrete examples of  such structures were found in \cite{KPT} and in later papers which further studied ``KPT theory''. For example, the theory of any ordered random hypergraph is extremely amenable.
\item 
%By the discussion following Corollary  \ref{corollary: amenability implies relative definable amenability}, whenever a  Fra\"{i}ss\'{e} class with canonical amalgamation has $\omega$-categorical Fra\"{i}ss\'{e} limit, then the theory of this limit is extremely amenable. For example, the theory of any random hypergraph is extremely amenable.
By Proposition \ref{proposition: canonical amalgamation}, whenever a  Fra\"{i}ss\'{e} class with canonical amalgamation over $\emptyset$ has $\omega$-categorical Fra\"{i}ss\'{e} limit, then the theory of this limit is extremely amenable. For example, the theory of any random hypergraph is extremely amenable.
\item By Corollary \ref{corollary:NIP characterization}, all NIP theories for which $\emptyset$ is an extension base are amenable; in particular, all stable, o-minimal, and C-minimal theories are amenable. But o-minimal theories are even extremely amenable, because any global non-forking extension of a given type over $\emptyset$ is $\autf_L(\C)$-invariant by Fact \ref{fact: charact. of non-forking in NIP theories}, and, on the other hand, in o-minimal theories, $\autf_L(\C)= \aut(\C)$ by \cite[Lemma 24]{Zi}. A stable theory is extremely amenable if and only of all [finitary] complete types over $\emptyset$ are stationary.
%udi
\item The theories of all measurable structures in the sense of Elwes and  Macpherson (e.g. pseudo-finite fields, 
smoothly approximable structures) are amenable by \cite[Remark 3.8(5)]{ElMa} and Corollary \ref{corollary: equivalent definitions of amenable theory}(4). Those theories are supersimple of finite D-rank by Corollaries 3.6 and 3.7 of \cite{ElMa}.

\item Let $\mathcal{C}$ be a class of finite structures of size at least $2$ in a language $L$ containing constants $c_1,c_2$ interpreted as distinct elements in all structures in $\mathcal{C}$.  Then the structures $A \in \mathcal{C}$ can be canonically expanded to $L'$-structures $A'$ with $\aut(A')=\aut(A)$, such that letting $\mathcal{C}':=\{A': A \in \mathcal{C}\}$, $T':=\Th(\mathcal{C}')$ is amenable in the sense that all its completions are amenable; and remains so over any finite set.  (It suffices to close the language of $\mathcal{C}$ under cardinality comparison quantifiers $Q\bar x \bar y(\phi(\bar x,\bar u),\psi(\bar y, \bar v))$,  asserting in finite models that there are at least as many tuples  $\bar x$ with $\phi(\bar x,\bar u)$ as  tuples $\bar y$ with $\psi(\bar y,\bar v)$, where $\bar x,\bar u,\bar y,\bar v$ are finite tuples of variables;  see the first page of \cite[Section 8.3]{ch-h}.  (Though the general quantifiers referred to here are actually only briefly mentioned before passing to a more specialized and more effective version appropriate there.)
It is easy to see (using constants $c_1,c_2$) that using these quantifiers, one can also express the relation $p |D(\bar a)| \geq q |D'(\bar b)|$ as $R_{p,q,D,D'}(\bar a,\bar b)$ for some $L'$-formula $R_{p,q,D,D'}(\bar x,\bar y)$ (without extra parameters),  where 
$p,q$ are positive integers and $D$, $D'$ are definable families of definable sets.
%; hence   the pseudo-finite measures relative to any given definable set are uniformly definable in the sense of continuous logic, thus the theory $T'$ is definably amenable.)
Hence, for any $\bar a$-definable set $D=D(\bar a)$ in a monster model $\C$ of any completion of $T'$, we can define an $\aut(\C/\bar a)$-invariant Keisler measure on definable subsets of $D$ by: $\mu(E(\bar b)):=\inf \{p/q : \C \models R_{p,q,D,E}(\bar a,\bar b)\}$. Thus,  the theory $T'$ is amenable. In fact, it is even definably amenable in the sense of Definition \ref{definition: definably amenable theory} below.)
 
%%udi

\item 
%In a joint paper of the second author with J. Lee and S. Moconja entitled ``Ramsey theory and topological dynamics for first order theories'' 
In \cite{KLM} (which was written about a year after the original preprint containing the material from this paper), Ramsey-theoretic characterizations of [extreme] amenability as well as various other dynamical properties of first order theories are established. Also, examples of amenable theories illustrating some other important phenomena (which we will not mention here) are given there. Let us only say that e.g. Examples 5.10 and 5.11 from that paper yield some amenable theories which are not extremely amenable, not NIP, and supersimple of SU-rank 1. Expanding them by an ``independent'' dense linear order, we get examples of amenable theories which are not extremely amenable, not NIP, and not simple. Namely, one can easily show the following. 

Let $T$ be the theory in a relational language $\{R,E_n,\leq\}_{n<\omega}$ (where all symbols are binary) saying that $R$ is irreflexive and symmetric, each $E_n$ is an equivalence relation with at least two classes and $E_0$ has finitely many classes,
%with a given finite (greater than $1$) number of classes depending on $n$, 
$\leq$ is a dense linear order, and the relations $R$, $\{E_n\}_{n<\omega}$, $\leq$ are ``independent'' in the sense that 
%any finite collection of formulas of the form $\pm S(x,a)$, where $S \in \{ R, E_n\}_{n<\omega}$ and $a$ is arbitrary, has a non-empty intersection with any open $\leq$-interval. 
the intersection of 
\begin{enumerate}
\item any finite collection of formulas of the form $\pm R(x,a)$ (with pairwise distinct $a$'s) with 
\item any finite collection $C_0,\dots,C_{n-1}$ of classes of the relations $E_0,\dots,E_{n-1}$, respectively, with
\item any open $\leq$-interval 
\end{enumerate}
is non-empty.
Then $T$ is complete with quantifier elimination, amenable (as in the proof of amenability in \cite[Example 5.10]{KLM}, and see also the next bullet), but not extremely amenable (because each global 1-type determines an $E_0$-class which can be moved to another $E_0$-class by some automorphism), not NIP (because the reduct to $R$ is the random graph), and not simple (because of the dense linear order $\leq$). 
%One can also generalize this example by replacing the theory  of independent equivalence relations $\{E_n\}_n$ with given numbers of classes by some other amenable theories, and then expand them ``independently'' by a random graph and a dense linear order, but we will discuss it here.
\item One can extend the last example as follows. Let us start from any amenable theory $T_0$ for which $\acl(A)=A$ for every $A$ (in place of the theory of independent equivalence relations $\{E_n\}_n$). Then add independently a random graph and a dense linear order to obtain a new theory $T$ (by ``independently'' we mean that the intersection of any infinite definable set with any collection of finitely many formulas of the form $\pm R(x,a)$ (with pairwise distinct $a$'s) and with any open interval is non-empty). Then $T$ is amenable, non-NIP, and non-simple.
%and it is extremely amenable if and only if $T_0$ is. 

The idea of the proof is as follows. We may assume that $T_0$ has q.e (by taking Morleyzation), and then so does $T$ (by a back-and-forth argument, using  randomness and the assumption that $\acl(A)=A$). Take any finitary type $p(\bar x) \in S_{\bar x}^{T}(\emptyset)$ (in the theory $T$), where $\bar x= (x_i)_{i<m}$; let $p_0 \in S_{\bar x}^{T_0}(\emptyset)$ be its restriction in the theory $T_0$.  Let $\C$ be a monster model of both $T$ and $T_0$. Consider the partial type $\pi (\bar x) := \{ x_i \geq c: i<m, c\in \C\} \cup \{R(x_i,a): i<m, a\in \C\}$ in $T$, and let $X:=[\pi(\bar x)]$ be the induced closed subset of $S^T_p(\C)$. By randomness and the extra assumption on $\acl$, $X$ is non-empty, so it is an $\aut^{T}(\C)$-flow. Let $\Phi \colon S_p^T(\C) \to S_{p_0}^{T_0}(\C)$ be the restriction map. By q.e., randomness, and the assumption on $\acl$, one easily checks that $\Phi |_{X} \colon X \to S_{p_0}^{T_0}(\C)$ is a monomorphism of $\aut^T(\C)$-flows whose image is the $\aut^{T_0}(\C)$-subflow $Y$ of $S_{p_0}^{T_0}(\C)$ given by $\{ x_i \ne c: i<m, c\in \C\}$. Since $T_0$ is amenable, $Y$ carries an $\aut^{T_0}(\C)$-invariant, Borel probability measure. The pullback of this measure under $\Phi |_{X}$ is an $\aut^T(\C)$-invariant, Borel probability measure on $X$, which witnesses that $p(\bar x)$ is amenable. Thus, $T$ is amenable.
\end{itemize}

%\subsection{Amenability implies G-compactness:  the case of definable measures}\label{Subsection 4.1 and a half}
\section{Amenability implies $G$-compactness:  the case of definable measures}\label{Subsection 4.1 and a half}

%%%Krzys: I changed the sentence below.
%The main result of this section of the paper, that amenability of $T$ implies that $T$ is $G$-compact, will be proved in full generality in Section \ref{Subsection 4.2}.  
Theorem \ref{The main theorem} will be proved  in full generality in Section \ref{Subsection 4.2}.  However, some special cases have a relatively easy proof. 
One such  is the NIP case above. Another case is when  $T$ is extremely amenable, where the proof of Remark 4.21 of \cite{KrPi} shows that in fact $T$ is $G$-trivial (the Lascar group is trivial).  This is made explicit in Proposition \ref{proposition: extremely amenable local} below.  Ivanov's observation in \cite{Ivanov} that if $T$ is strongly determined over $\emptyset$, then Lascar strong types coincide with (Shelah) strong types follows from Proposition \ref{proposition: extremely amenable local} by working over $\acl^{eq}(\emptyset)$.  However, deducing $G$-compactness of $T$ from  amenability of $T$ in general  is more complicated, and the proof in Section \ref{Subsection 4.2} uses 
%the full strength of the machinery developed in Section 2.  
a version of the stabilizer theorem (i.e. Corollary 2.12 of \cite{HKP1}) and requires adaptations of some ideas from Section 2 of \cite{HKP1} involving various computations concerning relatively definable subsets of $\aut(\C)$.
%%%Krzys Jan 2021: I added ``(see Definition \ref{definition: definably amenable theory} below)''.
This section is devoted to a proof of the main result in the special case when amenability of $T$ is witnessed by $\emptyset$-definable, global Keisler measures, rather than just $\emptyset$-invariant Keisler measures (see Definition \ref{definition: definably amenable theory} below). 
%%%Krzys: I changed the next sentences.
%We will make use of CL-stability as in Section 3 of \cite{HKP1}. But this time we will also make explicit use of results from \cite{BY-U}. 
We will make use of continuous logic stability as in Section 3 of \cite{HKP1}. 
%%%Krzys Dec 2020 final: I replaced the next sentence by the text proposed by Anand.
%But this time we will also make explicit use of results from \cite{BY-U}. 
We want to clarify that we are not trying to give here an introduction to
continuous logic in the sense of the precise formalism of \cite{BY-U}. In \cite{HKP1}, we
gave a self contained account of a certain approach to continuous logic in the
context of classical first order theories.  Here, we discuss, among other
things, compatibilities of our approach in \cite{HKP1} with the specific formalism
of \cite{BY-U}, as we want to make explicit use of results from \cite{BY-U}.

%%%Krzys: I added the next paragraph.
Recall the standard notion of a definable function from a model to a compact, Hausdorff space (the equivalent statements given below follow from \cite[Lemma 3.2]{GPP}). A function $f \colon M^n \to C$ (where $M$ is a model and $C$ is a compact, Hausdorff space) is called {\em definable} if the preimages under $f$ of any two disjoint closed subsets of $C$ can be separated by a definable subset of $M^n$; equivalently, $f$ is induced by a (unique) continuous map from $S_n(M)$ to $C$. This is equivalent to the condition that $f$ has a (unique) extension to an {\em $M$-definable} function $\hat{f}\colon \C^n \to C$ (where $\C$ is a monster model), meaning that the preimages under $\hat{f}$ of all closed subsets of $C$ are type-definable over $M$. A function from  $\C^n$ to $C$ is said to be $A$-definable, if the preimages of all closed subsets are type-definable over $A$. In particular, a Keisler measure $\mu(\bar x)$ is said to be {\em $\emptyset$-definable} if for every formula $\varphi(\bar x, \bar y)$, the function $\mu(\varphi(\bar x, \bar y)) : \C^{\bar y} \to [0,1]$ is $\emptyset$-definable.

%%%Krzys Jan 2021:  I added the next definition and the paragraph following it (according to Udi's suggestion).
\begin{dfn}\label{definition: definably amenable theory}
A theory $T$ is {\em definably amenable} if every formula $\varphi(\bar x)$ extends to a $\emptyset$-definable, global Keisler measure.
\end{dfn}

We are aware that, while there is no formal clash, this is a different use
of the adverb than in the case of definable groups;  ``definably'' there
refers to the measure algebra, not to the measure .  Thus ``definably definably amenable'' would express, in
addition, that the measure $\mu$ on $\Def(G)$ is itself definable.   In the case
of theories, the inner qualifier is redundant.

%%%Krzys: Instead of "our formalism" I wrote "our formalism from Section 3 of \cite{HKP1}" 
We first discuss the relationship between our formalism from Section 3 of \cite{HKP1} and that of \cite{BY-U}.  Start with our (classical) complete first order theory $T$, which we assume for convenience to be $1$-sorted.  This is  a theory in continuous logic in the sense of \cite{BY-U}, but where the metric is discrete and all relation symbols are $\{0,1\}$ valued, where $0$ is treated as ``true'' and $1$ as ``false''.  The type spaces $S_{n}(T)$ are of course Stone spaces. 
%%%Krzys: I added the next three sentences.
Recall from Definition 3.4 of \cite{HKP1} that by a {\em continuous logic (CL) formula over $A$} we mean a continuous function $\phi \colon S_n(A)\to \R$.  If $\phi$ is such a CL-formula, then for any $\bar b\in M^n$ (where $M \models T$) by $\phi(\bar b)$ we mean $\phi(\tp(\bar b/A))$. So CL-formulas over $A$ can be thought of as $A$-definable maps from $\C^n$ to compact subsets of $\R$ (note that the range of every CL-formula is compact).
 What are called {\em definable predicates}, in finitely many variables and without parameters,  in \cite{BY-U} are precisely CL-formulas over $\emptyset$ in our sense, but where the range is contained in $[0,1]$.  Namely, a definable predicate in $n$ variables is given by a continuous function from $S_{n}(T)$ to $[0,1]$.  The CL-generalization of Morleyizing $T$ consists of adding all such definable predicates as new predicate symbols in the sense of continuous logic.  So if $M$ is a model of $T$ and $\phi({\bar x})$ is such a new predicate symbol, then the interpretation $\phi(M)$ of $\phi$ in $M$ is the function taking an $n$-tuple $\bar a$ from $M$ to $\phi(\tp({\bar a}))$.  
%Let us call this new theory $T_{CL}$ (a theory of continuous logic with quantifier elimination, see Definition), to which we can apply the results of \cite{BY-U}. As just remarked, any  model $M$ of $T$ expands uniquely to a model of $T_{CL}$, but we will still call it $M$.
Let us call this new theory $T_{CL}$ (a theory of continuous logic), to which we can apply the results of \cite{BY-U}. 
%Note that $T_{CL}$ does not depend on the choice of $M$ (by induction on the complexity of CL-formulas and using uniform approximation of definable predicates by formulas, one shows that if $M \prec N$ are models of $T$, then $M\prec_{CL} N$). 
By the discussion after Proposition 3.10 in \cite{BY-U}, one sees that $T_{CL}$ has quantifier elimination \cite[Definition 4.14]{BY-U}. 
As just remarked, any model $M$ of $T$ expands uniquely to a model of $T_{CL}$, but we will still call it $M$. Note also that any saturation or homogeneity property of $M$ is preserved under passing to $T_{CL}$ (which follows from the observations that the group of automorphisms of $M$ is preserved and types in $S^T(A)$ determine types in $S^{T_{CL}}(A)$ for any $A$).

To understand {\em imaginaries} as in Section 5 of \cite{BY-U},  we have to also consider definable predicates, without parameters, but in possibly infinitely (yet countably) many variables.  As in Proposition 3.10 of \cite{BY-U}, such a definable predicate in infinitely  many variables can be identified with a continuous function from $S_{\omega}(T)$ to $[0,1]$, where $S_{\omega}(T)$ is the space of complete types of $T$ in a fixed countable sequence of variables.  We feel free to call such a function (and the corresponding function on $\omega$-tuples  in  models of $T$ to $[0,1]$)  a CL-formula in infinitely many variables. Let us now fix  a definable predicate (so CL-formula) $\phi({\bar x},{\bar y})$, where 
${\bar x}$ is a finite tuple of variables, and $\bar y$ is a possibly infinite (but countable) sequence of variables. A ``code'' for  the CL-formula (with parameters ${\bar a}$ and finite tuple ${\bar x}$ of free variables) $\phi({\bar x},{\bar a})$ will then be a CL-imaginary in the sense of \cite{BY-U},  and all CL-imaginaries will arise in this way.
The  precise formalism (involving new sorts with their own distance relation) is not so important, but the point is that the code will be something fixed by precisely those automorphisms 
(of a saturated model) which fix the formula $\phi({\bar x},{\bar a})$.  More precisely, the code will be the equivalence class of ${\bar a}$ with respect to the obvious equivalence 
relation $E_{\phi}({\bar y},{\bar z})$, on tuples of the appropriate length. If ${\bar y}$ is a finite tuple of variables, then we will call a corresponding imaginary (i.e. code for $\phi({\bar x},{\bar a})$) a {\em finitary CL-imaginary}.  We will work in the saturated model $\bar M=\C$ of $T$ which will also be a saturated model of $T_{CL}$. When we speak about interdefinability of various objects, we mean a priori in the sense of automorphisms of ${\bar M}$, 
%%%Krzys Dec 2020: I added the explanation below, as requested by the referee.
i.e. two objects are {\em interdefinable} if they are preserved by exactly the same automorphisms of $\bar M$.

The notion of hyperimaginaries is well-established in (usual, classical)  model theory \cite{LaPi}.  A {\em hyperimaginary} is by definition $\bar a/E$, where $\bar a$ is a possibly infinite (but small compared with the saturation) tuple and $E$ a type-definable over $\emptyset$ equivalence relation on tuples of the relevant size.  
%It is known that up to interdefinability we may restrict to tuples of length at most $\omega$ (see \cite[Remark 3.1.8]{Wag}), which we henceforth do. When the length of ${\bar a}$ is finite, we call ${\bar a}/E$ a {\em finitary hyperimaginary}.  
Up to interdefinability we may restrict to tuples of length at most $\omega$ (see \cite[Remark 3.1.8]{Wag}). When the length of ${\bar a}$ is finite, we call ${\bar a}/E$ a {\em finitary hyperimaginary}.
%We need the following consequence of \cite[Lemma 4.18]{LaPi}: 

\begin{rem}\label{remark: from 4.18 of LaPi}
If $E$ is a type-definable over $\emptyset$ equivalence relation and the $\aut(\bar M)$-orbit of $\bar a/E$ is bounded, then there is a {\em bounded} type-definable over $\emptyset$ equivalence relation $F$ refining $\equiv$ which agrees with $E$ on $[\bar a]_\equiv$.
\end{rem}

\begin{proof}
By \cite[Lemma 4.18]{LaPi}, $F:=E_{KP} \cup (E \cap ([\bar a]_\equiv \times [\bar a]_\equiv))$ works, because $E \cap ([\bar a]_\equiv \times [\bar a]_\equiv)$ is $\emptyset$-type-definable and bounded by assumption. (More precisely, \cite[Lemma 4.18]{LaPi} is stated for finite tuples but works the same for infinite tuples, too).
\end{proof}

The following is routine, but we sketch the proof.

\begin{lem}\label{lemma: CL-imaginaries interdefinable with hyperimaginaries}  (i) Any [finitary] CL-imaginary is interdefinable with a  [finitary]  hyperimaginary.
%(ii) Any  (finitary) hyperimaginary is in the definable closure of  a sequence of  (finitary)  CL-imaginaries. 
%Any finitary  hyperimaginary is interdefinable with a sequence of CL-imaginaries. 
\newline
(ii) If $E$ is a bounded, type-definable over $\emptyset$ equivalence relation, then each class of $E$ is interdefinable with a sequence of finitary CL-imaginaries.
\end{lem}

\begin{proof}
(i) If $\phi({\bar x},{\bar y})$ is a CL-formula where ${\bar y}$ is a possibly countably infinite tuple,  then the equivalence relation $E({\bar y}, {\bar z})$ which says of $({\bar b}, {\bar c})$ that the functions $\phi({\bar x}, {\bar b})$ and $\phi({\bar x}, {\bar c})$ are the same is a type-definable over $\emptyset$ equivalence relation in $T$. (Indeed, consider any $\bar b,\bar c$ such that $E(\bar b,\bar c)$ does not hold, i.e. for some $\bar a$, $\phi(\bar a,\bar b) \ne \phi(\bar a, \bar c)$. Then there are formulas $\varphi(\bar x,\bar y) \in \tp(\bar a\bar b/\emptyset)$ and $\psi(\bar x,\bar y) \in \tp(\bar a\bar c/\emptyset)$ such that whenever $\models \varphi(\bar a',\bar b') \wedge \psi(\bar a',\bar c')$, then $\phi(\bar a',\bar b') \ne \phi(\bar a', \bar c')$. Put $\theta(\bar y,\bar z):=\exists \bar x \varphi(\bar x,\bar y) \wedge \psi(\bar x,\bar z)$. Then $\models \theta(\bar b,\bar c)$ and whenever  $\models \theta(\bar b',\bar c')$, then $\phi(\bar a',\bar b') \ne \phi(\bar a', \bar c')$ for some/any $\bar a'$ such that $\models \varphi(\bar a',\bar b') \wedge \psi(\bar a',\bar c')$. So we have shown that the complement of $E$ is $\bigvee$-definable over $\emptyset$.)
%\newline
%(ii) Let $E({\bar x}, {\bar y})$ be a type-definable over $\emptyset$ equivalence relation on possibly infinite (but countable) tuples.  It is well known that $E$ is equivalent to a conjunction %of type-definable over $\emptyset$ equivalence relations, each defined by a countable collection of formulas. So we may assume $E$ is defined by a countable  collection of $\lL$-formulas 
%$\{\phi_{i}({\bar x}, {\bar y})\}_{i<\omega}$ where in addition each $\phi_{i}$ is symmetric and reflexive, 
%$(\forall \bar x)(\forall \bar y)(\varphi_0(\bar x,\bar y))$, 
%and $\phi_{i+1}({\bar x}, {\bar y})\wedge\phi_{i+1}({\bar y}, {\bar z})$ implies $\phi_{i}({\bar x}, {\bar z})$.  Let $\psi({\bar x},{\bar y})$ be $0$, or $\frac{1}{n+1}$ where $n$ is %greatest 
%such that there exist $\bar x'$ and $\bar y'$ with $\phi_n(\bar x, \bar x') \wedge \phi_n(\bar y, \bar y') \wedge \phi_{n}({\bar x'}, {\bar y'})$. 
%such that $\phi_{n}({\bar x}, {\bar y})$.
%Then $\psi$ is a formula of $T_{CL}$ and for any ${\bar a}$, the hyperimaginary ${\bar a}/E$ is in the definable closure of the code of $\psi({\bar x},{\bar a})$. 
%interdefinable with the code of $\psi({\bar x},{\bar a})$. 
\newline
%%%Krzys Dec 2020: I replaced $\psi$ by $d_E$, because the referee said that $\psi$ is a too generic notation.
(ii) By \cite[Theorem 4.15, Corollary 1.5]{LaPi} and Remark \ref{remark: from 4.18 of LaPi}, without loss of generality $E$ lives on finite tuples. 
It is well-known that $E$ is equivalent to a conjunction of equivalence relations each of which is defined by a countable collection of formulas over $\emptyset$ and is also bounded (see \cite[Lemma 3.1.3]{Wag}). So we may assume that $E$ is defined by a countable collection of formulas. Then $\C/E$ is a compact space, metrizable via an $\aut(\C)$-invariant metric $d$ (see \cite[Section 3, p. 237]{KrNe}). Define $d_E(\bar x, \bar y):= d(\bar x/E, \bar y/E)$. This is clearly a CL-formula, and we see that each $\bar a/E$ is interdefinable with the code of $d_E(\bar x, \bar a)$.
\end{proof}

%%%Krzys Dec 2020: The referee asked why below we wrote $Aut(\bar M)$ and not $Aut(M)$. I did not add any explanation. Boundedness does not make sense in small models.
Let $\acl_{CL}^{eq}(\emptyset)$ denote the  collection of CL-imaginaries which have a bounded number of conjugates under $\aut({\bar M})$. Likewise  $\bdd^{heq}(\emptyset)$ is the collection of hyperimaginaries with a bounded number of conjugates under $\aut(\bar M)$.  
%Now, Theorem 4.15 of \cite{LaPi} says that any bounded hyperimaginary is interdefinable with a sequence of finitary bounded hyperimaginaries.  
By Lemma \ref{lemma: CL-imaginaries interdefinable with hyperimaginaries} and Remark \ref{remark: from 4.18 of LaPi}, we get

%%%Krzys Dec 2020: The referee asked us to make the next corollary more precise. In my opinion, it is precise, so I did not make any changes here.
\begin{cor}\label{corollary: CL-acl = bdd} 
(i) Up to interdefinability, $\acl_{CL}^{eq}(\emptyset)$ coincides with $\bdd^{heq}(\emptyset)$. 
\newline
(ii) Moreover, $\acl_{CL}^{eq}(\emptyset)$ is interdefinable with the collection of finitary CL-imaginaries with a bounded number of conjugates under $\aut({\bar M})$. 
\end{cor}

We now appeal to the local stability results in \cite{BY-U} (which go somewhat beyond what we deduced purely from Grothendieck in Section 3 of \cite{HKP1}).  
%%%Krzys: I recalled the definition of stability of a CL-formula.
%%%Krzys Dec 2020: I replaced everywhere in this section $\Delta$ by $\Delta_{\textrm{st}}$ as requested by the referee.
Fix a finite tuple ${\bar x}$ of variables  and consider $\Delta_{\textrm{st}}({\bar x})$, the collection of all stable formulas  (without parameters) $\phi({\bar x},{\bar y})$ of $T_{CL}$, where $\bar y$ varies and where stability 
%is as defined in Section \ref{Section 3}. 
of  $\phi(\bar x, \bar y)$ means that for all $\epsilon > 0$ there do not exist $\bar a_{i}, \bar b_{i}$ for $i<\omega$ (in the monster model) such that for all $i<j$, $|\phi(\bar a_{i},\bar b_{j}) - \phi(\bar a_{j},\bar b_{i})| \geq \epsilon$. 
%%%Krzys Dec 2020: I added the next equivalent definition, as the referee asked us to justify why $d_E$ in the proof of Remark 3.3 is stable, and it is convenient to use there this equivalent definition.
(By Ramsey theorem and compactness, $\phi(\bar x, \bar y)$ is stable if and only if whenever $(\bar a_i,\bar b_i)_{i<\omega}$ is indiscernible, then $\phi(\bar a_i,\bar b_j)=\phi(\bar a_j, \bar b_i)$ for $i<j$.)
For an $n$-tuple $\bar b$ and set $A$ of parameters (including possibly CL-imaginaries), 
$\tp_{\Delta_{\textrm{st}}}({\bar b}/A)$ is the function taking the formula  $\phi({\bar x},{\bar a})$ to  $\phi({\bar b},{\bar a})$, where $\phi({\bar x},{\bar y})\in \Delta_{\textrm{st}}$ and $\phi({\bar x},{\bar a})$ is over $A$ (i.e invariant under $\aut({\bar M}/A)$).  By definition, a {\em complete $\Delta_{\textrm{st}}$-type over $A$} is something of the form $\tp_{\Delta_{\textrm{st}}}({\bar b}/A)$ (and $\bar b$ is a realization of it).

\begin{rem}\label{remark: types over bdd coincide with types over CL-acl}
%%%Krzys: I added an explanation what it means that the types in these two senses coincide. 
For any ${\bar b}$, $\tp({\bar b}/\bdd^{heq}(\emptyset))$ (in the classical case) coincides with $\tp_{\Delta_{\textrm{st}}}({\bar b}/\acl_{CL}^{eq}(\emptyset))$ in the continuous framework, meaning that  $\tp({\bar b}/\bdd^{heq}(\emptyset)) =  \tp({\bar b'}/\bdd^{heq}(\emptyset))$ if and only if $\tp_{\Delta_{\textrm{st}}}({\bar b}/\acl_{CL}^{eq}(\emptyset)) =\tp_{\Delta_{\textrm{st}}}({\bar b'}/\acl_{CL}^{eq}(\emptyset))$.
\end{rem}

\begin{proof} Using Corollary \ref{corollary: CL-acl = bdd},  the left hand side always implies the right hand side.  
%For the other direction, assume that ${\bar a}/E$ is a (finitary)  bounded hyperimaginary.  As in the proof of Lemma \ref{lemma: CL-imaginaries interdefinable with hyperimaginaries}(ii), we may assume that $E$ is defined by a suitable countable collection $\{\phi_i({\bar x}, {\bar y}):i<\omega\}$ of $L$-formulas.  We may also assume that $E$ is itself bounded (i.e. has boundedly many classes).  Let the CL-formula $\psi({\bar x}, {\bar y})$ be as in the proof of Lemma \ref{lemma: CL-imaginaries interdefinable with hyperimaginaries}.   As $E$ is bounded, $\psi({\bar x}, {\bar y})$ is stable.   The code of $\psi({\bar x},{\bar a})$ is clearly in $\acl_{CL}^{eq}(\emptyset)$ (so $\psi({\bar x}, {\bar a})$ is over $\acl_{CL}^{eq}(\emptyset)$)  and clearly ${\bar a}/E$ is in the definable closure of this code.
%
For the other direction, since $\bar x \equiv_{\bdd^{heq}(\emptyset)} \bar y$ is a bounded, type-definable over $\emptyset$ equivalence relation (in fact, it is exactly $E_{KP}$), it is enough to show that for any  bounded, type-definable over $\emptyset$ equivalence relation $E$, whenever $\tp_{\Delta_{\textrm{st}}}({\bar b}/\acl_{CL}^{eq}(\emptyset)) = \tp_{\Delta_{\textrm{st}}}({\bar b'}/\acl_{CL}^{eq}(\emptyset))$, then $E(\bar b, \bar b')$. 
%As in the proof of Lemma \ref{lemma: CL-imaginaries interdefinable with hyperimaginaries}(ii), we may assume that $E$ is defined by a suitable countable collection $\{\phi_i({\bar x}, {\bar y})\}_{i<\omega}$ of $\lL$-formulas. Let the CL-formula $\psi({\bar x}, {\bar y})$ be as in the proof of Lemma \ref{lemma: CL-imaginaries interdefinable with hyperimaginaries}.   As $E$ is bounded, $\psi({\bar x}, {\bar y})$ is stable. The code of $\psi({\bar x},{\bar b})$ is interdefinable with $\bar b /E$, hence it is in  $\acl_{CL}^{eq}(\emptyset)$, and so $\psi({\bar x}, {\bar a})$ is over $\acl_{CL}^{eq}(\emptyset)$. Since clearly $\psi(\bar b, \bar b)=0$, we conclude that $\psi(\bar b', \bar b)=0$ which means that $E(\bar b, \bar b')$.
%%%Krzys Dec 2020: I added the next sentence, and then replaced $\psi$ by $d_E$. But I did not repeat the definition of $d_E$ outside the proof of Lemma 3.1(ii) (although the referee wanted us to do that).
Without loss of generality we may assume that $E$ is defined by countably many formulas.
Let $d_E(\bar x, \bar y)$ be the CL-formula from the proof of Lemma  \ref{lemma: CL-imaginaries interdefinable with hyperimaginaries}(ii). 
%%%Krzys Dec 2020: I added a justification of stability of $d_E$, as the referee asked for it.
As $E$ is bounded, $d_E({\bar x}, {\bar y})$ is stable (because for every indiscernible sequence $(\bar a_i,\bar b_i)_{i<\omega}$ all $\bar a_i$'s are in a single $E$-class and all $\bar b_i$'s are in a single $E$-class, and so $d_E(\bar a_i, \bar b_j)=d(\bar a_i/E,\bar b_j/E)$ is constant for all $i,j<\omega$). The code of $d_E({\bar x},{\bar b})$ is interdefinable with $\bar b /E$, hence it is in  $\acl_{CL}^{eq}(\emptyset)$, 
%%%Krzys: Below $\psi({\bar x}, {\bar b})$ in place of $\psi({\bar x}, {\bar a})$.
and so $d_E({\bar x}, {\bar b})$ is over $\acl_{CL}^{eq}(\emptyset)$. Since clearly $d_E(\bar b, \bar b)=0$, we conclude that $d_E(\bar b', \bar b)=0$ which means that $E(\bar b, \bar b')$.
\end{proof} 

%%%Krzys: I explained below how we can view complete \phi-types over $M$. Then I gave the definition of definability of types. Please check it.
% We have already explained in Section \ref{Section 3} what we mean by definability of a complete $\Delta_{\textrm{st}}$-type over a model $M$.  The following is a consequence of the local theory developed in Section 7 of \cite{BY-U} and the discussion around glueing in Section 8 of the same paper. We restrict ourselves to the case needed, i.e. over $\emptyset$. 
If $M$ is a model, then $p=\tp_{\Delta_{\textrm{st}}}({\bar b}/M)$ can be identified with the collection of functions $f_\phi\colon M^n \to \R$ taking $\bar a \in M^n$ to $\phi(\bar b, \bar a)$, for $\phi(\bar x, \bar y) \in \Delta_{\textrm{st}}$. The type $\tp_{\Delta_{\textrm{st}}}({\bar b}/M)$ is said to be {\em definable} (over $M$) if the functions $f_\phi$ are induced by CL-formulas over $M$; it is {\em definable over $A$} if the $f_\phi$'s are induced by CL-formulas over $A$.
%it is {\em definable over $A$} if the codes of the CL-formulas $f_\phi$ are invariant over $A$.
A {\em $\varphi(\bar x, \bar y)$-definition of $p$} is a CL-formula $\chi(\bar y)$ such that $\varphi(\bar b, \bar a) = \chi(\bar a)$ for all $\bar a$ from $M$.

The following is a consequence of the local theory developed in Section 7 of \cite{BY-U} and the discussion around gluing in Section 8 of the same paper (see the proof of \cite[Proposition 8.7]{BY-U}). We restrict ourselves to the case needed, i.e. over $\emptyset$. 

%%%Krzys Dec 2020: I added CL as requested by the referee, although I think it is clear that we talk about CL-types here.
\begin{fct}\label{fact: stationarity of types over CL-acl}  Let $p({\bar x})$ be a (CL-)complete $\Delta_{\textrm{st}}$-type over $\acl_{CL}^{eq}(\emptyset)$. Then for any model $M$  (which note contains $\acl_{CL}^{eq}(\emptyset)$) there is a unique complete $\Delta_{\textrm{st}}$-type $q(\bar x)$ over $M$ such that $q(\bar x)$ extends $p(\bar x)$ and $q$ is definable over $\acl_{CL}^{eq}(\emptyset)$. We say $q = p|M$. In particular, if $M \prec N$, then $p|M$ is precisely the restriction of $p|N$ to $M$. 
\end{fct} 

%%%Krzys Dec 2020: The referee asked us to re-write this definition, but I do not see any point to do that. So did not do that.
\begin{dfn}\label{definition: stable independence} 
We say that ${\bar b}$ is {\em stably independent} from $B$  (or that ${\bar b}$ and $B$ are {\em stably independent}) if $\tp_{\Delta_{\textrm{st}}}({\bar b}/B)$ equals the restriction of $p|M$ to $B$, where $M$ is some model containing $B$ and $p = \tp_{\Delta_{\textrm{st}}}({\bar b}/\acl_{CL}^{eq}(\emptyset))$. 
\end{dfn}

Stable independence is clearly invariant under automorphisms. The usual Erd\"{o}s-Rado arguments, together with Fact \ref{fact: stationarity of types over CL-acl} give:

%%%Krzys Dec 2020: I replaced $p(\bar x)$ by $q$ in the next corollary, because the usage of $p$ here and in Corollary 3.7 was confusing for the referee.
\begin{cor}\label{corollary: existence of invariant stably independent sequences} 
Let  $q$ be a complete $\Delta_{\textrm{st}}$-type over $\acl_{CL}^{eq}(\emptyset)$. Then there is an infinite sequence $({\bar b_{i}}:i<\omega)$ of realizations of $q$ which is indiscernible and such that ${\bar b_{i}}$ is stably independent from $\{{\bar b}_{j}: j< i\}$ for all $i$.
\end{cor}

The following consequence of Fact \ref{fact: stationarity of types over CL-acl} will also be important for us.

%%%Krzys: In the statement below and in its proof, I replaced $\bdd^{heq}(A)$ by $\bdd^{heq}(\emptyset)$.
%%%Krzys Dec 2020: I replaced $\bar b, \bar c$ by $\bar a, \bar b$, as $\bar c$ is an enumeration of $\C$ (usually in this paper).
\begin{cor}\label{corollary: by stationarity from Udi's note} Suppose we have finite tuples ${\bar a}$ and ${\bar b}$ from the (classical) model $\C$. Suppose that ${\bar a}$ is stably independent from ${\bar b}$.  Then for any stable CL-formula $\psi({\bar x}, {\bar y})$ (over $\emptyset$), the value of $\psi({\bar a}, {\bar b})$ depends only on $\tp({\bar a}/\bdd^{heq}(\emptyset))$ and $\tp({\bar b}/\bdd^{heq}(\emptyset))$ (in the sense of the classical structure $\C$). 
\end{cor}

\begin{proof}  
%%%Krzys Dec 2020: I changed the first sentence of the proof.
%Let $p({\bar x}) = \tp({\bar b}/\bdd^{heq}(\emptyset))$, which by Remark \ref{remark: types over bdd coincide with types over CL-acl}  coincides with $\tp_{\Delta_{\textrm{st}}}({\bar b}/\acl_{CL}^{eq}(\emptyset))$.  
Let $p({\bar x}):= \tp_{\Delta_{\textrm{st}}}({\bar a}/\acl_{CL}^{eq}(\emptyset))$.
%%%Krzys Dec 2020: I removed ``= \bdd^{heq}(\emptyset)$'', cause it is only up to interdefinability.
The $\psi({\bar x},{\bar y})$-type of $p|\C$ is by Fact \ref{fact: stationarity of types over CL-acl} definable by  a CL-formula $\chi(\bar y)$ over $\acl_{CL}^{eq}(\emptyset)$.  So assuming the stable independence of $\bar a$ and 
$\bar b$, by definition and Fact \ref{fact: stationarity of types over CL-acl}, the value of $\psi({\bar a},{\bar b})$ is equal to $\chi({\bar b})$, which by Remark \ref{remark: types over bdd coincide with types over CL-acl} depends only on $\tp({\bar b}/\bdd^{heq}(\emptyset))$. 
%%%Krzys Dec 2020: I added ``(equivalently, $\bar a' \models \tp(\bar a/\bdd^{heq}(\emptyset))$)''.
If ${\bar a}$ is replaced by another realization ${\bar a}'$ of $p$ (equivalently, $\bar a' \models \tp(\bar a/\bdd^{heq}(\emptyset))$) which is stably independent from another realization ${\bar b}'$ of $\tp({\bar b}/\bdd^{heq}(\emptyset))$, then  the above shows that 
$\psi({\bar a}',{\bar b}') = \chi({\bar b}')= \chi({\bar b}) = \psi({\bar a},{\bar b})$. 
\end{proof}

\begin{prop}\label{proposition: amalgamation of wide types}
Let $\mu=\mu_{\bar x}$ be a global, $\emptyset$-definable Keisler measure.  Let ${\bar a}$ and ${\bar b}$ be tuples of the same length from $\C$, with the same type over $\bdd^{heq}(\emptyset)$, and stably independent.  Let $p({\bar x},{\bar a})$ be a complete type over $\bar a$ which is ``$\mu$-wide'' in the sense that every formula in $p({\bar x},{\bar a})$ gets $\mu$-measure $> 0$. Then  the partial type $p({\bar x},{\bar a})\cup p({\bar x},{\bar b})$ is also $\mu$-wide (again in the sense that every formula implied by it has $\mu$-measure $>0$). 
\end{prop}

\begin{proof} By definition, we have to show that if $\phi({\bar x},{\bar a})$ is a formula with $\mu$-measure $> 0$, then  $\phi({\bar x},{\bar a})\wedge\phi({\bar x},{\bar b})$ has 
$\mu$-measure $> 0$.  By $\emptyset$-definability of $\mu$, the function $\psi({\bar y},{\bar z})$ defined to be  $\mu(\phi({\bar x},{\bar y})\wedge\phi({\bar x}, {\bar z}))$ is 
definable over $\emptyset$, i.e. is a CL-formula without parameters. Moreover, by Proposition 2.25 of \cite{Hrushovski-approximate},  $\psi({\bar y},{\bar z})$ is stable. 
%%%Krzys Dec 2020: I added ``q:='' because of the referee's confusion.
Bearing in mind Remark \ref{remark: types over bdd coincide with types over CL-acl}, let, by Corollary \ref{corollary: existence of invariant stably independent sequences},  $({\bar a}_{i}:i<\omega)$ be an indiscernible sequence of realizations of $q:=\tp({\bar a}/\bdd^{heq}(\emptyset))$  such that ${\bar a}_{j}$ and ${\bar a}_{i}$ are stably independent for all $i<j$ (equivalently for some $i<j$).   Since $\mu$ is $\aut(\C)$-invariant, we see that $\mu(\phi({\bar x},{\bar a}_{i}))$ is positive and  constant for all $i$, and $\mu(\phi({\bar x},{\bar a}_{i})\wedge \phi({\bar x},{\bar a}_{j}))$ is positive (and constant) for $i\neq j$.  In particular, $\psi({\bar a}_{0}, {\bar a}_{1}) > 0$.  By Corollary \ref{corollary: by stationarity from Udi's note}, $\psi({\bar a}, {\bar b}) > 0$, which is what we had to prove. 
\end{proof} 

\begin{prop}\label{proposition: main theorem for definable measures}
%%%Krzys Jan 2021: I replaced the first two sentences of this proposition by ``Suppose that (the classical, first order theory) $T$ is definably amenable.''.
%Suppose that amenability of (the classical, first order theory) $T$ is witnessed by $\emptyset$-definable Keisler measures. Namely, for every formula $\phi({\bar x})$ over $\emptyset$ there is a global $\emptyset$-definable Keisler measure $\mu_{\bar x}$ concentrating on $\phi({\bar x})$.  
Suppose that (the classical, first order theory) $T$ is definably amenable. Then $T$ is $G$-compact. 
%%%Krzys Dec 2020: I added the next sentence, which follows from the proof below.
In fact, the diameter of each Lascar strong type (over $\emptyset$) is bounded by $2$.
\end{prop}

%%%Krzys Dec 2020: I replaced everywhere in the proof below $\bar b, \bar c$ by $\bar a, \bar b$. In the first sentence I added ``more precisely, the Lascar distance between them is $\leq 2$''.
\begin{proof} We have to show that if  ${\bar a}$, ${\bar b}$ are tuples of the same (but possibly infinite) length and with the same type over $\bdd^{heq}(\emptyset)$, then they  have the same Lascar strong type; more precisely, the Lascar distance between them is $\leq 2$.

Observe that, by compactness, without loss of generality, we can and do assume that $\bar a$ and $\bar b$ are finite tuples (this is because if we show that all corresponding finite subtuples of $\bar a$ and $\bar b$ are at Lascar distance at most 2, then so are $\bar a$ and $\bar b$).

Assume first that ${\bar a}$ and ${\bar b}$ are stably independent in the sense of Definition \ref{definition: stable independence}.  
Fix a model $M_{0}$ and enumerate it. We will find a copy $M$ of $M_{0}$ such that $\tp({\bar a}/M) = \tp({\bar b}/M)$ (which immediately yields that ${\bar a}$ and ${\bar b}$ have the same Lascar strong type; in fact, $d_L(\bar a, \bar b) \leq 1$).  
By compactness, given a consistent formula $\phi({\bar y})$ in finitely many variables, it suffices to find some realization ${\bar m}$ of $\phi({\bar y})$ such that $\tp({\bar a}/{\bar m}) = \tp ({\bar b}/{\bar m})$. 
%Again by compactness, we may assume that ${\bar a}$, ${\bar b}$ are finite tuples.  
By assumption, let $\mu_{\bar y}$ be a  $\emptyset$-definable, global Keisler measure concentrating on $\phi({\bar y})$. 
Let $p({\bar y},{\bar a})$ be a complete type over ${\bar a}$ which is $\mu$-wide. By Proposition \ref{proposition: amalgamation of wide types}, $p({\bar y},{\bar a})\cup p({\bar y},{\bar b})$ is also $\mu$-wide, in particular consistent. So let ${\bar m}$ realize it. 

In general, given finite tuples ${\bar a}$, ${\bar b}$ with the same type over $\bdd^{heq}(\emptyset)$, let ${\bar d}$ have the same type over $\bdd^{heq}(\emptyset)$ and be stably independent from $\{{\bar a}, {\bar b}\}$ (by  Remark \ref{remark: types over bdd coincide with types over CL-acl} and Fact \ref{fact: stationarity of types over CL-acl}). 
By what we have just shown,  $d_L({\bar a},{\bar d}) \leq 1$ and $d_L({\bar b},{\bar d}) \leq 1$. So $d_L({\bar a},{\bar b}) \leq 2$. 
\end{proof}

%\subsection{Amenability implies G-compactness: the general case}\label{Subsection 4.2}
\section{Amenability implies $G$-compactness: the general case}\label{Subsection 4.2}

Let $T$ be an arbitrary theory, $\C \models T$ a monster model, and $\bar c$ an enumeration of $\C$.
%%%Krzys: I changed the next sentence.
The goal of this section is to prove Theorem \ref{The main theorem}; in fact, we will get more precise information:

\begin{thm}\label{theorem: relative definable amenability implies G-compactness}
If $T$ is amenable, then $T$ is $G$-compact. In fact, the diameter of each Lascar strong type (over $\emptyset$) is bounded by $4$.
%If $p(\bar x) \in S(\emptyset)$ is amenable, then the Lascar strong types on $p(\bar x)$ coincide with Kim-Pillay strong types. In particular, if $T$ is amenable, then $T$ is G-compact.
\end{thm}

%%%Krzys: I removed the next paragraph.
%Recall, once again, that in \cite{KrPi} it was deduced from Conjecture \ref{conjecture: Conjecture 0.4 of KrPi} for groups with a basis of open sets at 1 consisting of open subgroups that if $M$ is a countable $\omega$-categorical structure and $\aut(M)$ is amenable (as a topological group), then $\Th(M)$ is G-compact. By Corollary \ref{corollary: amenability implies relative definable amenability}, we see that Theorem \ref{theorem: relative definable amenability implies G-compactness} is a generalization of this result.

Before we start our analysis towards the proof of Theorem \ref{theorem: relative definable amenability implies G-compactness}, let us first note the analogous statement for extreme amenability, which is much easier to prove.

%%%Krzys Dec 2020; I did not remove the ``In particular'' part, as I do not agree with the referee. Also, I did not move it to Section 2.
\begin{prop}\label{proposition: extremely amenable local}
If $p(\bar x) \in S(\emptyset)$ is extremely amenable, then $p(\bar x)$ is a single Lascar strong type. Moreover, the Lascar diameter of $p(\bar x)$ is at most $2$.

In particular, if $T$ is extremely amenable, then the Lascar strong types coincide with complete types (over $\emptyset$), i.e. the Lascar Galois group $\gal_L(T)$ is trivial. Furthermore, if $T$ is extremely amenable, the Lascar distance between any two elements which have the same type is at most 1.
\end{prop}

\begin{proof}
Choose $\C$ so that $\bar x$ is short in $\C$. Let $q \in S_p(\C)$ be invariant under $\aut(\C)$. Fix $\bar \alpha \models q$ (in a bigger model). Take a small $M \prec \C$ and choose $\bar \beta \in \C$ such that $\bar \beta \models q|_M$. Then $\bar \alpha \, E_L\, \bar \beta$. But also, for any $\sigma \in \aut(\C)$, $\sigma(\bar \beta) \models \sigma(q)|_{\sigma[M]} = q|_{\sigma[M]}$, and so $\sigma(\bar \beta) \, E_L \, \bar \alpha$. Therefore,  $\sigma(\bar \beta) \, E_L\,  \bar \beta$ for any $\sigma \in \aut(\C)$, which shows that $p(\bar x)$ is a single Lascar strong type.

For the ``moreover part'' notice that, in the above argument, both $d_L(\bar \alpha, \bar \beta)$ and $d_L(\sigma(\bar \beta), \bar \alpha)$ are bounded by $1$.

The ``in particular part'' follows immediately from the first part, but we also give a shorter proof suggested by the referee which yields additionally Lascar distance at most 1. Consider any $\bar a \equiv \bar b$. By assumption, there exists a model $M$ enumerated as $\bar m$ such that $\tp(\bar m/\C)$ is invariant. Then $\bar a \equiv_M \bar b$, so $d_L(\bar a,\bar b) \leq 1$. 
\end{proof}

%%%Krzys Dec 2020: I shortened the next pargraph, and referred to the new Def. 2.14 instead of Cor. 2.15.
%Recall from Corollary \ref{corollary: measure extends to relatively type-definable sets} that by a {\em relatively type-definable subset} of $\aut(\C)$ we mean a subset of the form $A_{\pi,\bar a,\bar b}:= \{ \sigma \in \aut(\C) : \C \models \pi(\sigma(\bar a), \bar b))\}$ for some partial type $\pi(\bar x, \bar y)$ (without parameters), where $\bar x$ and $\bar y$ are short tuples of variables and $\bar a$, $\bar b$ are from $\C$. (Note that although here we allow repetitions in the tuple $\bar a$, whereas in Corollary \ref{corollary: measure extends to relatively type-definable sets} $\bar a$ was a subtuple of $\bar c$, both versions yield the same class of relatively type-definable sets.) Without loss $\bar x$ is of the same length as $\bar y$ and $\bar a =\bar b$, and then we write $A_{\bar \pi, \bar a}$. In fact, the following remark is very easy.
Recall from Definition \ref{definition: relatively type-definable sets} that by a {\em relatively type-definable subset} of $\aut(\C)$ we mean a subset of the form $$A_{\pi,\bar a,\bar b}:= \{ \sigma \in \aut(\C) : \C \models \pi(\sigma(\bar a), \bar b))\}$$ 
for some partial type $\pi(\bar x, \bar y)$ (without parameters), where $\bar x$ and $\bar y$ are short tuples of variables and $\bar a$, $\bar b$ are from $\C$.  Without loss $\bar x$ is of the same length as $\bar y$ and $\bar a =\bar b$, and then we write $A_{\pi, \bar a}$. In fact, the following remark is very easy.

\begin{rem}
For any partial types $\pi_1(\bar x_1, \bar y_1)$ and $\pi_2(\bar x_2,\bar y_2)$ and tuples $\bar a_1, \bar a_2, \bar b_1, \bar b_2$ in $\C$ corresponding to $\bar x_1,\bar x_2,\bar y_1,\bar y_2$, one can find partial types $ \pi_1'(\bar x, \bar y)$ and $\pi_2'(\bar x, \bar y)$ with $\bar x$ of the same length (by which we also mean of the same sorts) as $\bar y$ and a tuple $\bar a$ in $\C$ corresponding to $\bar x$ such that $A_{\pi_1,\bar a_1,\bar b_1}=A_{\pi_1',\bar a}$ and $A_{ \pi_2,\bar a_2,\bar b_2}=A_{ \pi_2',\bar a}$.
\end{rem}

For a short tuple $\bar \alpha$ and a short tuple of parameters $\bar b$, a subset of $\aut(\C)$ is called {\em relatively $\bar \alpha$-type-definable over $\bar b$} if it is of the form $A_{\pi,\bar \alpha,\bar b}$ for some partial type $\pi(\bar x, \bar y)$.

The next fact was observed in \cite{KrNeSi}.

\begin{fct}[Proposition 5.2 of \cite{KrNeSi}]\label{fact: boundedness of the index implies AutfL(C) is contained}
If $G$ is a closed, bounded index subgroup of $\aut(\C)$ (with $\aut(\C)$ equipped with the pointwise convergence topology), then $\autf_L(\C) \leq G$.
\end{fct}

Using an argument similar to the proof  of Fact \ref{fact: boundedness of the index implies AutfL(C) is contained}, we will first show

\begin{prop}\label{proposition: bounded index implies containment of Autf-KP}
If $G$ is a relatively type-definable, bounded index subgroup of $\aut(\C)$, then $\autf_{KP}(\C) \leq G$.
\end{prop}

\begin{proof}
Let $\sigma_i$, $i<\lambda$, be a set of representatives of the left cosets of $G$ in $\aut(\C)$ (so $\lambda$ is bounded). Then 
$$G':= \bigcap_{\sigma \in \aut(\C)} G^\sigma = \bigcap_{i<\lambda} G^{\sigma_i}$$
is a normal, bounded index subgroup  of $\aut(\C)$ (where $G^\sigma :=\sigma G \sigma^{-1}$).

Let us show now that $G'$ is relatively type-definable. We have $G = A_{\pi,\bar a}=\{ \sigma \in \aut(\C) : \C \models \pi(\sigma(\bar a), \bar a)\}$ for some type $\pi(\bar x, \bar y)$ (with short $\bar x$, $\bar y$) and tuple $\bar a$ in $\C$. Then $G^{\sigma_i}= \{\sigma \in \aut(\C) : \C \models \pi(\sigma(\sigma_i(\bar a)), \sigma_i(\bar a))\}$, so putting $\bar a '= \langle \sigma_i (\bar a) \rangle_{i <\lambda}$, $\bar x' = \langle \bar x_i \rangle_{i < \lambda}$,  $\bar y' = \langle \bar y_i \rangle_{i <\lambda}$ (where $\bar x_i$ and $\bar y_i$ are copies of $\bar x$ and $\bar y$, respectively) and $\pi'(\bar x',\bar y') = \bigcup_{i<\lambda} \pi(\bar x_i,\bar y_i)$ (as a set of formulas), we see that
$$(*)\;\;\;\;\;\;\; G' =  A_{\pi',\bar a'}= \{\sigma \in \aut(\C) : \C \models \pi'(\sigma(\bar a'), \bar a')\},$$
which is clearly relatively type-definable.

The orbit equivalence relation $E$ of the action of $G'$ on the set of realizations of $\tp(\bar a'/\emptyset)$ is a bounded equivalence relation. This relation is type-definable, because 
$$\bar \alpha \, E\,  \bar \beta \iff (\exists g \in G') (g (\bar \alpha) =\bar \beta) \iff (\exists \bar b')( \pi'(\bar b', \bar a') \wedge \bar a' \bar \alpha \equiv \bar b' \bar \beta).$$
But $E$ is also invariant (as $G'$ is a normal subgroup of $\aut(\C)$), so $E$ is type-definable over $\emptyset$. Therefore, $E$ is refined by $E_{KP}$.

Now, take any $\sigma \in \autf_{KP}(\C)$. By the last conclusion, there is $\tau \in G'$ such that $\sigma(\bar a')=\tau(\bar a')$. Then $\tau^{-1}\sigma(\bar a')=\bar a'$ and $\sigma = \tau (\tau^{-1}\sigma)$. 
%Since the above formula for $G'$ shows that $G'\cdot \Fix(\bar a')= G'$, we get $\sigma \in G'$. Thus, $\autf_{KP}(\C) \leq G' \leq G$.
Since $(*)$ implies that $G'\cdot \Fix(\bar a')= G'$, we get $\sigma \in G'$. Thus, $\autf_{KP}(\C) \leq G' \leq G$.
\end{proof}

Recall that a subset $C$ of a group is called {\em (left) generic} if finitely many left translates of it covers the whole group; $C$ is called {\em symmetric} if it contains the neutral element and $C^{-1}=C$.

\begin{cor}\label{corollary: Autf-KP contained in the intersection}
If $\{C_i : i \in \omega\}$ is a family of relatively definable, generic, symmetric subsets of $\aut(\C)$ such that $C_{i+1}^2 \subseteq C_i$ for all $i \in \omega$, then $\bigcap_{i \in \omega} C_i$ is a subgroup of $\aut(\C)$  containing $\autf_{KP}(\C)$.  
\end{cor}

\begin{proof}
It is clear that  $\bigcap_{i \in \omega} C_i$ is a subgroup of $\aut(\C)$, and it is easy to show that it has bounded index (at most $2^{\aleph_0}$). Moreover, it is clearly relatively type-definable. Thus, the fact that it contains $\autf_{KP}(\C)$ follows from Proposition \ref{proposition: bounded index implies containment of Autf-KP}.
\end{proof}

%The following lemma will be used several times.

\begin{lem}\label{lemma: relative type-definability is preserved under products}
i) Let $\pi(\bar x, \bar y)$ be a partial type (over $\emptyset$) and $\bar a$, $\bar b$ short tuples from $\C$ corresponding to $\bar x$ and $\bar y$, respectively. Then $A_{\pi,\bar a,\bar b}^{-1}=A_{\pi',\bar b, \bar a}$, where $\pi'(\bar y, \bar x) :=\pi(\bar x, \bar y)$ (i.e. the type $\pi$ with the exchanged roles of variables).\\
ii) Let $n\geq 2$ be a natural number.  Let $\bar x$, $\bar y$ and $\bar x_1,\dots, \bar x_n$ be disjoint, short tuples of variables of the same length. Then there exists a partial type $\Phi_n(\bar x, \bar y, \bar x_1, \dots, \bar x_n)$ such that for every partial types $\pi_1(\bar x_1, \bar y), \dots, \pi_n(\bar x_n, \bar y)$ and tuple $\bar a$ corresponding to $\bar x$ one has 
$$A_{\pi_1, \bar a} \cdot \ldots \cdot A_{\pi_n,\bar a} = A_{\pi, \bar a},$$
where 
$$\pi(\bar x, \bar y) := (\exists \bar x_1, \dots, \bar x_n) (\pi_1(\bar x_1,\bar y) \wedge \dots \wedge \pi_n(\bar x_n,\bar y) \wedge \Phi_n(\bar x, \bar y, \bar x_1,\dots,\bar x_n)).$$
\end{lem}

\begin{proof}
(i) follows immediately from the fact that for any $\sigma \in \aut(\C)$
$$  \C \models \pi(\sigma(\bar a), \bar b) \iff \C \models \pi(\bar a, \sigma^{-1}(\bar b)) \iff \C \models \pi'(\sigma^{-1}(\bar b), \bar a).$$
%$A_{\pi,\bar a,\bar b}^{-1} = \{ \sigma \in \aut(\C) : \C \models \pi(\sigma(\bar a), \bar b) \}^{-1} = \{ \sigma \in \aut(\C}
(ii) We will show that for $n=2$ the type $\Phi_2(\bar x, \bar y, \bar x_1,\bar x_2) := (\bar x \bar x_1 \equiv \bar x_2 \bar y)$ and for $n \geq 3$ the type $\Phi_n(\bar x, \bar y, \bar x_1, \dots, \bar x_n)$ defined as
$$(\exists \bar z_1,\dots,\bar z_{n-2})(\bar x \bar z_{n-2} \equiv \bar x_n \bar y \wedge \bar z_{n-2} \bar z_{n-3} \equiv \bar x_{n-1}\bar y \wedge \dots \wedge \bar z_2 \bar z_1 \equiv \bar x_3 \bar y \wedge \bar z_1 \bar x_1 \equiv \bar x_2 \bar y)$$
is as required.

First, let us see that $A_{\pi_1, \bar a} \cdot \ldots \cdot A_{\pi_n,\bar a} \subseteq A_{\pi, \bar a}$. Take $\sigma$ from the left hand side, i.e. $\sigma = \sigma_1 \ldots \sigma_n$, where $\models \pi_i(\sigma_i(\bar a), \bar a)$. 
%Put  $\bar a_i := \sigma_i(\bar a)$ for $i=1,\dots,n$, and $\bar b_i := \sigma_1(\dots (\sigma_{i+1}(\bar a)\dots)$ for $i=1,\dots, n-2$.  
Then $\models \pi(\sigma(\bar a),\bar a)$ is witnessed by  $\bar x_i := \sigma_i(\bar a)$ for $i=1,\dots,n$ and $\bar z_i := (\sigma_1\dots \sigma_{i+1})(\bar a)$ for $i=1,\dots, n-2$. So $\sigma \in A_{\pi, \bar a}$.

Finally, we will justify that  $A_{\pi_1, \bar a} \cdot \ldots \cdot A_{\pi_n,\bar a} \supseteq A_{\pi, \bar a}$. 
%%%Krzys: I added "Consider the case $n \geq 3$."
Consider the case $n \geq 3$.
Take any $\sigma$ such that  $\models \pi(\sigma(\bar a),\bar a)$. Let $\bar a_1,\dots,\bar a_n$ be witnesses for $\bar x_1,\dots, \bar x_n$, and $\bar b_1,\dots, \bar b_{n-2}$ be witnesses for $\bar z_1,\dots,\bar z_{n-2}$, i.e.:
\begin{enumerate}
\item $\models \pi_i(\bar a_i,\bar a)$ for $i=1,\dots,n$, and 
\item $\sigma(\bar a) \bar b_{n-2} \equiv \bar a_n \bar a \wedge \bar b_{n-2} \bar b_{n-3} \equiv \bar a_{n-1}\bar a \wedge \dots \wedge \bar b_2 \bar b_1 \equiv \bar a_3 \bar a \wedge \bar b_1 \bar a_1 \equiv \bar a_2 \bar a$.
\end{enumerate}
By (2), there are $\tau_1,\dots,\tau_{n-1} \in \aut(\C)$ mapping the right hand sides of the equivalences in (2) to the left hand sides. 
Then $\tau_1(\bar a_n)=\sigma(\bar a)$, so $\tau_1^{-1}\sigma(\bar a)=\bar a_n$, so $\tau_1^{-1}\sigma \in A_{\pi_n,\bar a}$ by (1). Next, $\tau_1(\bar a)= \bar b_{n-2}=\tau_2 (\bar a_{n-1})$, so $\tau_2^{-1}\tau_1(\bar a) = \bar a_{n-1}$, so $\tau_2^{-1}\tau_1 \in A_{\pi_{n-1},\bar a}$ by (1). We continue in this way, obtaining in the last step: $\tau_{n-1}(\bar a)=\bar a_1$, so $\tau_{n-1} \in A_{\pi_{1},\bar a}$ by (1). Therefore, 
$$\sigma = \tau_{n-1} (\tau_{n-1}^{-1}\tau_{n-2}) \dots (\tau_2^{-1}\tau_1)(\tau_1^{-1}\sigma) \in A_{\pi_1, \bar a} \cdot \ldots \cdot A_{\pi_n,\bar a}.$$
%%%Krzys: I added the explanation below for $n=2$.
For $n=2$, in (2), we just have $\sigma(\bar a) \bar a_1 \equiv \bar a_2 \bar a$, so taking $\tau_1 \in \aut(\C)$ which maps $\bar a_2 \bar a$ to $\sigma(\bar a) \bar a_1$, we get $\tau_1^{-1}\sigma \in A_{\pi_2,\bar a}$ and  $\tau_{1} \in A_{\pi_{1},\bar a}$ , hence $\sigma \in A_{\pi_1, \bar a} \cdot A_{\pi_2, \bar a}$.
\end{proof}

\begin{cor}\label{corollary: compactness for relatively definable subsets}
Let $\pi_1(\bar x,\bar y), \dots, \pi_n(\bar x,\bar y)$ be partial types, $\bar a$ a tuple corresponding to $\bar x$ and $\bar y$, and $\epsilon_1,\dots,\epsilon_n \in \{-1,1\}$.\\
(i)  Then 
$$A_{\pi_1, \bar a}^{\epsilon_1} \cdot \ldots \cdot A_{\pi_n,\bar a}^{\epsilon_n} = \bigcap \left\{A_{\varphi_1, \bar a}^{\epsilon_1} \cdot \ldots \cdot A_{\varphi_n,\bar a}^{\epsilon_n} : \pi_1 \vdash \varphi_1,\dots, \pi_n \vdash \varphi_n\right\}.$$
ii) If $A_{\pi_1, \bar a}^{\epsilon_1} \cdot \ldots \cdot A_{\pi_n,\bar a}^{\epsilon_n}$ is contained in a relatively definable subset $A$ of $\aut(\C)$, then there are $\varphi_i(\bar x, \bar y)$ implied by $\pi_i(\bar x, \bar y)$ for $i=1,\dots,n$, such that  $A_{\varphi_1, \bar a}^{\epsilon_1} \cdot \ldots \cdot A_{\varphi_n,\bar a}^{\epsilon_n} \subseteq A$.
\end{cor}

%%%Krzys: I added a one-sentence proof.
%%%Krzys Dec 2020: The referee asked for a more detailed proof (he did not see how to prove it), so I added some details. But to be honest, I think this is a an example of a proof which is better to see in mind rather than to write it down.
\begin{proof}
This follows from Lemma \ref{lemma: relative type-definability is preserved under products}, using compactness and the fact that $\C$ is a monster model. But let us give some details.

By Lemma \ref{lemma: relative type-definability is preserved under products}(i), we can clearly assume that $\epsilon_i=1$ for all $i$. 
%The proofs of (i) and (ii) are very similar, so we will show only (ii), where an additional issue concerning variables appears.
Then item (i) follows directly  from Lemma \ref{lemma: relative type-definability is preserved under products}(ii).
So it remains to show item (ii).

%By adding dummy variables if necessary, we can find a formula $\psi(\bar x\bar x', \bar y,\bar y')$ so that
Take a formula $\psi(\bar x',\bar y')$ (where $\bar x'$ and $\bar y'$ are of the same length and are disjoint from both $\bar x$ and $\bar y$) and $\bar a'$ such that $A=A_{\psi, \bar a'}$. We can also treat $\psi$ as $\psi(\bar x \bar x', \bar y\bar y')$, and then $A=A_{\psi,\bar a \bar a'}$. Similarly, $\pi_i$ can be treated as $\pi_i(\bar x \bar x',\bar y \bar y')$, and then the original product $A_{\pi_1, \bar a} \cdot \ldots \cdot A_{\pi_n,\bar a}$ can be written as $A_{\pi_1, \bar a \bar a'} \cdot \ldots \cdot A_{\pi_n,\bar a \bar a'}$. By Lemma \ref{lemma: relative type-definability is preserved under products}(ii) and strong $\kappa$-homogeneity of $\C$, we get that the type
$$(\exists \bar x_1\bar x_1',\dots, \bar x_n\bar x_n') (\pi_1(\bar x_1x_1',\bar a \bar a') \wedge \dots \wedge \pi_n(\bar x_n \bar x_n',\bar a \bar a') \wedge \Phi_n(\bar x \bar x', \bar a \bar a', \bar x_1 \bar x_1',\dots,\bar x_n\bar x_n'))$$
in conjunction with $\bar x  \bar x' \equiv \bar a \bar a'$ implies the type $\psi(\bar x \bar x',\bar a \bar a')$. Hence, by compactness, each type $\pi_i(\bar x\bar x',\bar y \bar y')$ can be replaced by a formula $\varphi_i(\bar x \bar x',\bar y \bar y')$ implied by  $\pi_i(\bar x\bar x',\bar y \bar y')$ so that the above implication is still valid. Since the types $\pi_i$ use only variables $\bar x, \bar y$, the formulas $\varphi_i$ can also be chosen only in variables $\bar x,\bar y$. Then, by the above implication (with the $\pi_i$'s replaced by $\varphi_i$'s) and Lemma \ref{lemma: relative type-definability is preserved under products}(ii), we get that $A_{\varphi_1, \bar a} \cdot \ldots \cdot A_{\varphi_n,\bar a} = A_{\varphi_1, \bar a \bar a'} \cdot \ldots \cdot A_{\varphi_n,\bar a \bar a'} \subseteq A$.
\end{proof}

\begin{lem}\label{lemma: the intersection contained in Autf-L}
Let $p(\bar x) \in S(\emptyset)$ with $\bar x$ short, $q \in S_p(\C)$, $M \prec \C$ small, and $\bar \alpha \models q|_M$.
%Let $p(\bar x) = \in S(\emptyset)$ with a short $\bar x$, and let $q \in S_p(\C)$.  Then 
Then $A_{q|_{\bar \alpha},\bar \alpha}A_{q|_{\bar \alpha},\bar \alpha}A_{q|_{\bar \alpha},\bar \alpha}^{-1}A_{q|_{\bar \alpha},\bar \alpha}^{-1} \bar \alpha \subseteq \{\bar \beta \subset \C: d_L(\bar \alpha,\bar \beta) \leq 4\} \subseteq [\bar \alpha]_{E_L}$.
\end{lem}

\begin{proof}
Let us start from the following\\[-2mm]

\begin{clm}
For any $\bar \beta \models q|_{\bar \alpha}$, $d_L(\bar \beta, \bar \alpha) \leq 1$.
\end{clm}

\begin{clmproof}
Take $\bar \gamma \models q|_{M\bar \alpha}$. Then $d_L(\bar \gamma, \bar \alpha) \leq 1$, so the conclusion follows from the fact that $\bar \beta \equiv_{\bar \alpha} \bar \gamma$.  
\end{clmproof}

Now, consider any $\sigma_1,\sigma_2,\sigma_3,\sigma_4 \in A_{q|_{\bar \alpha},\bar \alpha}$. Then $\sigma_i(\bar \alpha) \models q|_{\bar \alpha}$, so, by the claim, we get $d_L(\sigma_i(\bar \alpha), \bar \alpha) \leq 1$. Therefore, $d_L(\sigma_4^{-1}(\bar \alpha), \bar \alpha) \leq 1$, so $d_L(\sigma_3^{-1}\sigma_4^{-1}(\bar \alpha), \sigma_3^{-1}(\bar \alpha)) \leq 1$, so $d_L(\sigma_3^{-1}\sigma_4^{-1}(\bar \alpha), \bar \alpha) \leq 2$, so $d_L(\sigma_2\sigma_3^{-1}\sigma_4^{-1}(\bar \alpha), \sigma_2(\bar \alpha)) \leq 2$, so $d_L(\sigma_2\sigma_3^{-1}\sigma_4^{-1}(\bar \alpha), \bar \alpha) \leq 3$, so  $d_L(\sigma_1\sigma_2\sigma_3^{-1}\sigma_4^{-1}(\bar \alpha), \sigma_1(\bar \alpha)) \leq 3$, so  $d_L(\sigma_1\sigma_2\sigma_3^{-1}\sigma_4^{-1}(\bar \alpha), \bar \alpha) \leq 4$.
\end{proof}

The proof of the next lemma uses a version of the stabilizer theorem obtained in \cite[Corollary 2.12]{HKP1}. We will not recall here all the terminology involved in  \cite[Corollary 2.12]{HKP1}; the reader may consult Subsections 2.1, 2.2, and 2.3 of \cite{HKP1}. Let us recall here the main things. 
$\lL_k^{\textrm{gen}}$, $ k \in \omega$, is a recursively defined notion of largeness of $\bigvee$-definable sets concentrated on a definable group $G$, which is invariant under left translations. Passing to a sufficiently saturated extension $\bar G$ of $G$, for any $\bigvee$-definable subset $Y$,  $\lL_0^{\textrm{gen}}(Y)$ means that $Y \ne \emptyset$ and $\lL_k^{\textrm{gen}}(Y)$ means precisely that $\{g \in \bar G: \lL_{k-1}^{\textrm{gen}}(gY \cap Y)\}$ is generic (i.e. finitely many left translates of it cover $\bar G$). Next, $\St_{\lL_k} (Y) := \{g: \lL_k(gY \cap Y) \}$ is an operator from the class of $\bigvee$-definable sets concentrated on $G$ to itself (see the paragraphs after Remark 2.3 in \cite{HKP1}). We would like to emphasize that $\St_{\lL_k} (Y)$ need not be a subgroup of $G$. (In \cite{HKP1}, we are more precise and consider the class of $\bigvee$-positively definable sets, which is essential in the applications in \cite{HKP1}, but here we do not care about positive definability).  We will need the following basic remark.

\begin{rem}\label{remark: St invariant under translations}
If a $\bigvee$-definable set $Y$ is invariant under left translations by the elements of some subgroup $H$ of $G$, then $\St_{\lL_k} (Y)$ is invariant under both left and right translations by the elements of $H$. 
\end{rem}

%%%Krzys Dec 2020: I also added the next comment and fact. The referee asked for something like that (I guess).
Instead of stating \cite[Corollary 2.12]{HKP1} in full generality and with the full power, we give a particular case, which is sufficient for our application.

\begin{fct}\label{fact: simplified Cor. 2.12 from KKP1}
%Let $G$ be a group. Let $\mathcal{A}$ be a Boolean algebra of subsets of $G$ closed under (group-theoretic) products and left translations by the elements of $G$. Let $A \in \mathcal{A}$, and $\mu$ be a left-invariant, finitely additive measure on $\mathcal{A}$ with $\mu(A)>0$. Then there exist $l \in \mathbb{N}_{>0}$ and $g_1,\dots,g_n \in G$ such that for $A':=A \cap g_1A \cap \dots \cap g_nA$, $S:=\St_{\lL_{l-1}^{\textrm{gen}}}(B')$ is generic, symmetric, and satisfies $S^{16} \subseteq AAA^{-1}A^{-1}$.
Let $G$ be a group, and $A \subseteq G$. Let $\mathcal{A}$ be a Boolean algebra of subsets of $G$ which is invariant under left translations and includes $A$ and all sets of the form $(g_1A \cap \dots \cap g_kA)A$ (with $k \geq 1$ and $g_1,\dots,g_k \in G$). Let $\mu$ be a left-invariant, finitely additive measure on $\mathcal{A}$ with $\mu(A)>0$. Then there exist $l \in \mathbb{N}_{>0}$ and $g_1,\dots,g_n \in G$ such that for $A':=A \cap g_1A \cap \dots \cap g_nA$, the set $S:=\St_{\lL_{l-1}^{\textrm{gen}}}(A')$ (computed with respect to $\Th(G,\cdot, A)$) is generic, symmetric, and satisfies $S^{16} \subseteq AAA^{-1}A^{-1}$.
%Let $G$ be a group. Let $\mathcal{A}$ be a Boolean algebra of subsets of $G$ closed under (group-theoretic) products and left translations by the elements of $G$. Let $A \in \mathcal{A}$, and $\mu$ be a left-invariant, finitely additive measure on $\mathcal{A}$ with $\mu(A)>0$. Then there exist $l \in \mathbb{N}_{>0}$ and $g_1,\dots,g_n \in G$ such that for $A':=A \cap g_1A \cap \dots \cap g_nA$, $S:=\St_{\lL_{l-1}^{\textrm{gen}}}(B')$ is generic, symmetric, and satisfies $S^{16} \subseteq AAA^{-1}A^{-1}$.
\end{fct}

\begin{proof}
Apply  \cite[Corollary 2.12]{HKP1}  for  $\mathcal{B}:=\{ A\}$, $N:=16$, ${\mathcal D}:=  {\mathcal A}$, and $m := \mu$. As a result we obtain $l \in \mathbb{N}_{>0}$ and $g_1,\dots,g_n \in G$ such that for $A':=A \cap g_1A \cap \dots \cap g_nA$, $S:=\St_{\lL_{l-1}^{\textrm{gen}}}(A')$ is generic as a $\bigvee$-definable set in $(G, \cdot, A)$, symmetric, and satisfies $S^{16} \subseteq AAA^{-1}A^{-1}$. In particular, $S$ is a generic subset of $G$, i.e. finitely many left translates of $S$ cover $G$.
\end{proof}

%%%Krzys Dec 2020: I slightly modified the formulation of the next lemma, because of the referee's complains. 
\begin{lem}\label{lemma:  Y8 contained in A'}
%Assume $\aut(\C)$ is relatively definably amenable. By Corollary \ref{corollary: measure extends to relatively type-definable sets}, take the induced $\aut(\C)$-invariant, finitely additive, probability measure $\mu$ on the Boolean algebra ${\mathcal A}$ generated by relatively type-definable subsets of $\aut(\C)$.
Assume $\aut(\C)$ is relatively amenable which is witnessed by an $\aut(\C)$-invariant, regular, Borel probability measure $\tilde{\mu}$ on $S_{\bar c}(\C)$. Let $\mu$ be the induced $\aut(\C)$-invariant, finitely additive, probability measure on the Boolean algebra ${\mathcal A}$ generated by relatively type-definable subsets of $\aut(\C)$, as described in Corollary \ref{corollary: measure extends to relatively type-definable sets}.
Suppose $A \subseteq \aut(\C)$ is relatively type-definable 
%[over $\bar \alpha$], generic, symmetric, 
with $\mu(A)>0$ and $AAA^{-1}A^{-1} \subseteq A'$ for some relatively definable $A'\subseteq \aut(\C)$. Then there exists a relatively type-definable, generic, symmetric $Y \subseteq \aut(\C)$ such that $Y^8 \subseteq A'$.
\end{lem}

\begin{proof}
By Lemma \ref{lemma: relative type-definability is preserved under products}, relatively type-definable sets are closed under taking products and inversions, and one can easily check that also under left translations.\\[-2mm]
%And similarly for the class of relatively $\bar \alpha$-type-definable sets, except that it is closed under taking inversions provided that this class is restricted to sets relatively $\bar \alpha$-type-definable over $\bar \alpha$.\\[-2mm]

%%%Krzys Dec 2020: The referee asked to rewrite Claim 1(2), but I think it is OK. I only replaced the second ``for some'' by ``where'', and replaced ``which is'' in the parentheses by ``more precisely'').
\begin{clm}
There exists a generic and symmetric set $S \subseteq \aut(\C)$ such that:
\begin{enumerate}
\item $S^{16} \subseteq AAA^{-1}A^{-1}$,
\item $S = \{ \sigma \in \aut(\C): \tp(\sigma(\bar a)/\bar a) \in {\mathcal P}\}$ for some ${\mathcal P} \subseteq S_{\bar a}(\bar a)$, where $\bar a$ is a short tuple (more precisely, a tuple of finitely many conjugates by elements of $\aut(\C)$ of the tuple over which $A$ is relatively type-definable). 
%In the variant of the lemma involving $\bar \alpha$, here $\bar a= \bar \alpha$.
\end{enumerate}
\end{clm}

\begin{clmproof}

%%%Krzys: I added what is $\mathcal{D}$ and $m$.
%Apply \cite[Corollary 2.12]{HKP1} for $G:=\aut(\C)$, $A$ from the statement of Lemma \ref{lemma:  Y8 contained in A'}, $\mathcal{B}:=\{ A\}$, and $N:=16$. 
%%%Krzys Dec 2020: I changed the beginning of this proof, referring to the new Fact 4.11.
%Apply \cite[Corollary 2.12]{HKP1} for $G:=\aut(\C)$, $A$ from the statement of Lemma \ref{lemma:  Y8 contained in A'}, $\mathcal{B}:=\{ A\}$, $N:=16$, ${\mathcal D}:=  {\mathcal A}$, and $m := \mu$. As a result, we obtain a set $B'=A \cap \sigma_1[A] \cap \dots \cap \sigma_n[A]$ for some $\sigma_i$'s in $\aut(\C)$ such that for some $l \in \mathbb{N}_{>0}$, $S:=\St_{\lL_{l-1}}(B')$ is generic, symmetric, and satisfies $S^{16} \subseteq AAA^{-1}A^{-1}$. 
We apply Fact \ref{fact: simplified Cor. 2.12 from KKP1} for $G:=\aut(\C)$ and $A$, $\mu$ from the statement of Lemma \ref{lemma:  Y8 contained in A'}. As a result, we obtain a set $B=A \cap \sigma_1[A] \cap \dots \cap \sigma_n[A]$ for some $\sigma_i$'s in $\aut(\C)$ such that for some $l \in \mathbb{N}_{>0}$, $S:=\St_{\lL_{l-1}}(B)$ is generic, symmetric, and satisfies $S^{16} \subseteq AAA^{-1}A^{-1}$. 
%
%Since $A$ is relatively type-definable over some $\bar \alpha$, so is $B'$, but over $\bar a:= \bar \alpha \sigma_1(\bar \alpha)\dots\sigma_n(\bar \alpha)$. 
%Since $A$ is relatively $\bar \alpha$-type-definable over $\bar \alpha$ for some short tuple $\bar \alpha$, so is $B'$, but over $\bar a:= \bar \alpha \sigma_1(\bar \alpha)\dots\sigma_n(\bar \alpha)$.
Since $A$ is relatively type-definable over some short tuple $\bar \alpha$, so is $B$, but over $\bar a:= \bar \alpha \sigma_1(\bar \alpha)\dots\sigma_n(\bar \alpha)$.
%%%Krzys: I changed the next two sentence.
%Hence, by the definition of $S$, we easily get that 
Hence, $\aut(\C/\bar a) \cdot B=B$. 
%%%Krzys Dec 2020: I referred to Remark 4.10 below. I also added a justification of the final conclusion, as requested by the referee.
%Therefore, by the property of $\St_{l-1}$ recalled before Lemma \ref{lemma:  Y8 contained in A'}, we get that 
Therefore, by Remark \ref{remark: St invariant under translations}, we get that 
$$\aut(\C/\bar a) \cdot S \cdot \aut(\C/\bar a) =  S,$$
which means that $S = \{ \sigma \in \aut(\C): \tp(\sigma(\bar a)/\bar a) \in {\mathcal P}\}$ for some ${\mathcal P} \subseteq S_{\bar a}(\bar a)$. To see the last thing, consider any $\sigma \in S$. We need to show that every $\tau \in \aut(\C)$ satisfying $\tp(\tau(\bar a)/\bar a) = \tp(\sigma(\bar a)/\bar a)$ belongs to $S$. For this note that there is $\tau' \in \aut(\C/\bar a)$ with $\tau(\bar a) = \tau' \sigma(\bar a)$. Then $\sigma^{-1}\tau'^{-1}\tau \in \aut(\C/\bar a)$ and clearly $\tau = \tau' \sigma (\sigma^{-1}\tau'^{-1}\tau)$. Hence, $\tau \in S$.
\end{clmproof}

Take any $p \in {\mathcal P}$. We can write $p=p(\bar x, \bar a)$ for the obvious complete type $p(\bar x, \bar y)$ over $\emptyset$. Then $(A_{p,\bar a} \cdot A_{p,\bar a}^{-1})^8 \subseteq (SS^{-1})^8 = S^{16} \subseteq AAA^{-1}A^{-1} \subseteq A'$. Hence, by Corollary \ref{corollary: compactness for relatively definable subsets}(ii), there is $\psi_p(\bar x, \bar y) \in p(\bar x,\bar y)$ for which $(A_{\psi_p,\bar a} \cdot A_{\psi_p,\bar a}^{-1})^8 \subseteq A'$.

Now, the complement of $\bigcup_{p \in {\mathcal P}} A_{\psi_p,\bar a}$ equals $\bigcap_{p \in {\mathcal P}} A_{\neg \psi_p, \bar a}$ which is clearly relatively type-definable. Thus, $\bigcup_{p \in {\mathcal P}} A_{\psi_p,\bar a} \in {\mathcal A}$. On the other hand, $S \subseteq \bigcup_{p \in {\mathcal P}} A_{\psi_p,\bar a}$ and $S$ being generic implies that $\bigcup_{p \in {\mathcal P}} A_{\psi_p,\bar a}$ is generic. Therefore, $\mu(\bigcup_{p \in {\mathcal P}} A_{\psi_p,\bar a}) >0$. 

%%%Krzys Dec 2020: Instead of ``Let'' I wrote ``Recall that''.
Recall that $\tilde{\mu}$ is the $\aut(\C)$-invariant, regular, Borel probability measure on $S_{\bar c}(\C)$ from which $\mu$ is induced. Then $\tilde{\mu}(\bigcup_{p \in {\mathcal P}} [\psi_p])>0$, so, by regularity, there is a compact $K \subseteq \bigcup_{p \in {\mathcal P}} [\psi_p]$ of positive measure. But $K$ is covered by finitely many clopen sets $[\psi_p]$ one of which must be of positive measure, i.e. $\tilde{\mu}([\psi_p])>0$ for some $p \in {\mathcal P}$. Then $\mu(A_{\psi_p,\bar a})>0$.
This implies that $Y:= A_{\psi_p,\bar a} \cdot A_{\psi_p,\bar a}^{-1}$ is generic (as otherwise there would exist an infinite family of pairwise disjoint left translates of $A_{\psi_p,\bar a}$ which would contradict the fact that $\mu$ is a left invariant probability measure), and it is clearly symmetric. By Lemma \ref{lemma: relative type-definability is preserved under products}, it is also relatively type-definable. Moreover, by the choice of $\psi_p$, $Y^8 \subseteq A'$, so we are done. 
\end{proof}

\begin{cor}\label{corollary: Autf-KP contained in A4}
Assume $\aut(\C)$ is relatively amenable. By Corollary \ref{corollary: measure extends to relatively type-definable sets}, take the induced $\aut(\C)$-invariant, finitely additive, probability measure $\mu$ on the Boolean algebra ${\mathcal A}$ generated by relatively type-definable subsets of $\aut(\C)$.
Suppose $A \subseteq \aut(\C)$ is relatively type-definable and $\mu(A)>0$. Then $\autf_{KP}(\C) \subseteq AAA^{-1}A^{-1}$.
\end{cor}

\begin{proof}
Take any $A'$ relatively definable, symmetric, and such that $AAA^{-1}A^{-1} \subseteq A'$. Put $C_0 := A'$.

By Lemma \ref{lemma:  Y8 contained in A'}, we obtain a relatively type-definable, generic, symmetric $Y$ such that $(Y^4)^2 \subseteq A'$. So, by Corollary \ref{corollary: compactness for relatively definable subsets}, there is a relatively definable, symmetric $Y'$ satisfying $Y^4 \subseteq Y'$ and $Y'^2 \subseteq A'$. Put $C_1 := Y'$.

Next, we apply Lemma \ref{lemma:  Y8 contained in A'} to $Y$ in place of $A$ and $Y'$ in place of $A'$, and we obtain a relatively type-definable, generic, symmetric $Z$ such that $(Z^4)^2 \subseteq Y'$.  So, by Corollary \ref{corollary: compactness for relatively definable subsets}, there is a relatively definable, symmetric $Z'$ satisfying $Z^4 \subseteq Z'$ and $Z'^2 \subseteq Y'$. Put $C_2 := Z'$.

%%%Krzys Dec 2020: At the end of the next paragraph, I removed the notation $A^4$.
Continuing in this way, we obtain a family $\{ C_i : i \in \omega\}$ of relatively definable, generic, symmetric subsets of $\aut(\C)$ such that $C_{i+1}^2 \subseteq C_i$ for every $i \in \omega$. By Corollary \ref{corollary: Autf-KP contained in the intersection}, $\autf_{KP}(\C) \subseteq \bigcap_{i \in \omega} C_i\subseteq A'$. Since $A'$ was an arbitrary relatively definable, symmetric set containing $AAA^{-1}A^{-1}$, we get  $\autf_{KP}(\C) \subseteq AAA^{-1}A^{-1}$.
\end{proof}

%We finish with a proof of the main result of this section.
%%%Krzys: I changed the next sentence.

%%%Krzys Dec 2020: I extended the next paragraph according to the referee's suggestion.
We have now all the ingredients to prove Theorem \ref{theorem: relative definable amenability implies G-compactness}. Our goal will be to show that for any short tuple $\bar \alpha$, $[\bar \alpha]_{E_{KP}} \subseteq [\bar \alpha]_{E_L}$ (which just means that the $\autf_{KP}(\C)$-orbit of $\bar \alpha$ is contained in $ [\bar \alpha]_{E_L}$).

\begin{proof}[Proof of Theorem \ref{theorem: relative definable amenability implies G-compactness}]
%Let $\mu$ be a measure on witnessing relative definable amenability.
Recall that $T$ being amenable means that there exists an $\aut(\C)$-invariant, regular, Borel probability measure $\tilde{\mu}$ on $S_{\bar c}(\C)$. 
%By Corollary \ref{corollary: measure extends to relatively type-definable sets}, a measure $\tilde{\mu}$ on $S_{\bar c}(\C)$ witnessing relative amenability of $\aut(\C)$ induces an $\aut(\C)$-invariant, finitely additive, probability measure $\mu$ on the Boolean algebra ${\mathcal A}$ generated by relatively type-definable subsets of $\aut(\C)$.
By Corollary \ref{corollary: measure extends to relatively type-definable sets}, $\tilde{\mu}$  induces an $\aut(\C)$-invariant, finitely additive, probability measure $\mu$ on the Boolean algebra ${\mathcal A}$ generated by relatively type-definable subsets of $\aut(\C)$.

Consider any $p(\bar x) = \tp(\bar \alpha/\emptyset) \in S(\emptyset)$ with a short subtuple 
%(with extra repetitions allowed) 
$\bar \alpha$ of $\bar c$. Choose a $\mu$-wide type $q \in S_p(\C)$, i.e. $\tilde{\mu}([\varphi(\bar x', \bar b)])>0$ (equivalently, $\mu(A_{\varphi,\bar \alpha,\bar b})>0$) for any $\varphi(\bar x, \bar b) \in q$ (where $\bar x'\supset \bar x$ is the tuple of variables corresponding to $\bar c$).
%, with the necessary repetitions added). 
Take a small model $M \prec \C$. 
%Applying an appropriate automorphism of $\C$ to $q$ and $M$, and using $\aut(\C)$-invariance of $\mu$, we can assume that $\bar \alpha \models q|_M$.
Replacing $\bar \alpha$ by a realization $\bar \alpha'$ of $q|_M$, we can assume that $\bar \alpha \models q|_M$ (because $\bar \alpha \equiv \bar \alpha'$ implies that $[\bar \alpha]_{E_{KP}} = [\bar \alpha]_{E_L}$ is equivalent to $[\bar \alpha']_{E_{KP}} = [\bar \alpha']_{E_L}$).

%Take any small model $M \prec \C$ and tuple $\bar \alpha \models q|_M$. 
Consider any $\varphi(\bar x, \bar \alpha) \in q|_{\bar \alpha}$. Then $\mu(A_{\varphi,\bar \alpha})>0$, so, by Corollary \ref{corollary: Autf-KP contained in A4}, we conclude that $\autf_{KP}(\C) \subseteq A_{\varphi,\bar \alpha}A_{\varphi,\bar \alpha}A_{\varphi,\bar \alpha}^{-1}A_{\varphi,\bar \alpha}^{-1}$. Therefore, by Corollary \ref{corollary: compactness for relatively definable subsets}(i), we get
$$\autf_{KP}(\C) \subseteq \bigcap_{\varphi(\bar x, \bar \alpha) \in q|_{\bar \alpha}}  A_{\varphi,\bar \alpha}A_{\varphi,\bar \alpha}A_{\varphi,\bar \alpha}^{-1}A_{\varphi,\bar \alpha}^{-1} =  A_{q|_{\bar \alpha},\bar \alpha}A_{q|_{\bar \alpha},\bar \alpha}A_{q|_{\bar \alpha},\bar \alpha}^{-1}A_{q|_{\bar \alpha},\bar \alpha}^{-1}.$$
On the other hand,  Lemma \ref{lemma: the intersection contained in Autf-L} tells us that 
$$A_{q|_{\bar \alpha},\bar \alpha}A_{q|_{\bar \alpha},\bar \alpha}A_{q|_{\bar \alpha},\bar \alpha}^{-1}A_{q|_{\bar \alpha},\bar \alpha}^{-1} \bar \alpha \subseteq \{ \bar \beta : d_L(\bar \alpha, \bar \beta) \leq 4\} \subseteq [\bar \alpha]_{E_L}.$$
Therefore, $[\bar \alpha]_{E_{KP}} = [\bar \alpha]_{E_L}$ has Lascar diameter at most $4$.
%, and then the same is clearly true for any other realization of $p$.
\end{proof}

Theorem \ref{theorem: relative definable amenability implies G-compactness} is a global result. It is natural to ask whether we can extend it to a local version (as in Proposition \ref{proposition: extremely amenable local}).

\begin{ques}\label{question: local case}
Is it true that if $p(\bar x) \in S(\emptyset)$ is amenable, then the Lascar strong types on $p(\bar x)$ coincide with Kim-Pillay strong types? Does amenability of $p(\bar x)$ imply that the Lascar diameter of $p(\bar x)$ is at most 4?
\end{ques}

One could think that the above arguments should yield the positive answer to these questions. The problem is that, assuming only amenability of $p(\bar x)$, we have the induced measure $\mu$ but defined only on the Boolean algebra of relatively $\bar \alpha$-type-definable subsets of $\aut(\C)$, for a fixed $\bar \alpha\models p$. So, for the recursive proof of Corollary \ref{corollary: Autf-KP contained in A4} to go through, starting from a set $A\subseteq \aut(\C)$ relatively $\bar \alpha$-type-definable [where for the purpose of answering Question \ref{question: local case} via an argument as in the proof of Theorem \ref{theorem: relative definable amenability implies G-compactness}, we can additionally assume that $A$ is defined over  $\bar \alpha$] of positive measure, we need to produce the desired $Y$ also relatively $\bar \alpha$-type-definable [over $\bar \alpha$] (in order be able to continue our recursion). But this requires a strengthening of Lemma \ref{lemma:  Y8 contained in A'} to the version where for $A$ relatively $\bar \alpha$-type-definable of positive measure one wants to obtain the desired $Y$ which is also relatively $\bar \alpha$-type-definable; the variant with $A$ and $Y$ defined over $\bar \alpha$ would also be sufficient. Trying to follow the lines of the proof of Lemma \ref{lemma:  Y8 contained in A'}, even if $A$ is defined over $\bar \alpha$, Claim 1 requires a longer tuple $\bar a$ which produces the desired set $Y$ which is relatively $\bar a$-type-definable, and this is the only obstacle to answer positively Question \ref{question: local case} via the above arguments.

Another question is whether the bound $4$ on the Lascar diameters of Lascar strong types in Theorem \ref{theorem: relative definable amenability implies G-compactness} could be decreased. Proposition \ref{proposition: main theorem for definable measures} tells us that it can be decreased to $2$ under the stronger assumption of definable amenability of $T$.

\end{document}